\documentclass[12pt,twoside]{amsart}
\input{includeNice3}
\usepackage{mabliautoref}
\usepackage{amssymb}
\usepackage{amscd}
\usepackage[abbrev,alphabetic]{amsrefs}
\usepackage{hyperref}
\usepackage[a4paper,margin=1in]{geometry}

\usepackage{xypic}
\usepackage{tikz-cd}
\usepackage{caption}

\usepackage{url}
\RequirePackage{booktabs, multirow}
\RequirePackage{pgf}

\title[Properness of the moduli space of stable surfaces]
{On the properness of the moduli space of stable surfaces over $\mathbb{Z}[1/30]$} 
\author{Emelie Arvidsson, Fabio Bernasconi and Zsolt Patakfalvi} 

\address{University of Utah Mathematics Department JWB 311,155 S 1400 E, Salt Lake City, UT 84112}
\email{arvidsson@math.utah.edu}

\address{Departement Mathematik und Informatik, Universit{\"a}t Basel, Spiegelgasse 1, CH-4051 Basel, Switzerland.} 
\email{fabio.bernasconi@unibas.ch}

\address{EPFL SB MATH CAG
	MA C3 615 (B\^atiment MA)
	Station 8
	CH-1015 Lausanne}
\email{zsolt.patakfalvi@epfl.ch}

\DeclareMathOperator{\cent}{centre}
\DeclareMathOperator{\cod}{codim}

\newcommand{\MO}{\mathcal{O}}

\newcommand{\Q}{\mathbb{Q}}
\newcommand{\Z}{\mathbb{Z}}

\begin{document}

\begin{abstract}
    We show the properness of the moduli stack of stable surfaces over $\mathbb{Z}[1/30]$, assuming the locally-stable reduction conjecture for stable surfaces.
	This relies on a local Kawamata--Viehweg vanishing theorem for 3-dimensional log canonical singularities at closed point of characteristic $p \neq 2, 3$ and $5$ which are not log canonical centres.
\end{abstract}
	
\maketitle

\tableofcontents
	
\section{Introduction}
	
    In \cite[Theorem 5.2]{DM69}, Deligne and Mumford proved that the moduli stack of stable curves $\overline{\mathcal{M}}_g$ of given genus $g \geq 2$ is a \emph{proper} Deligne-Mumford (DM) stack over $\mathbb{Z}$.
    By introducing stable curves (\emph{i.e.}, curves with at worst nodal singularities and ample canonical class) into the moduli problem, they were able to construct a natural compactification of the moduli of smooth curves of genus $g$, which led to interesting applications, such as the proof of irreducibility of $\mathcal{M}_g$ \cite{DM69}, and the proof of general semi-stable reduction for curves in \cite{de_Jong_Families_of_curves_and_alterations}.

    The natural higher dimensional generalisation of smooth curves of genus at least 2 are smooth canonically polarised varieties.
	Hence, it is natural to look for a compactification of the moduli space of these. A possible approach has been proposed by Koll\'ar and Shepherd-Barron in \cite{KSB88} using the Minimal Model Program (MMP for short).
    According to Koll\'ar and Shepherd-Barron, the correct generalisation of stable curves to arbitrary dimensions are \emph{stable varieties}, projective varieties with semi-log canonical singularities and ample canonical class.
    We refer to the book \cite{k-moduli} for a comprehensive treatment of the construction of the moduli space of stable varieties in characteristic 0.
	
	The case of positive and mixed characteristic presents further difficulties. To mention a few: the MMP is still largely conjectural in dimension $>3$, the invariance of plurigenera (even asymptotic) is known to fail \cite{Bri20, Kol22}, the singularities of the MMP are cohomologically more complicated \cite{CT19, Ber19} and other problems arise due to presence of inseparable morphisms \cite[Section 8.8]{k-moduli}.
	However the MMP for 3-folds in positive characteristic $p \geq 5$ and mixed characteristic $(0, p>5)$ has now been established \cite{HX15, CTX15, Bir16, BW17, DW22, HW22, TY23, 7authors} and, following the strategy in characteristic 0, many of the steps needed for the construction of the moduli space $\overline{\mathcal{M}}_{2,v}$ of stable surfaces have been proven in \cite{HK19, Pat17, 7authors,Pos21b, Pos21}. 
	Nowadays, we know that $\overline{\mathcal{M}}_{2,v}$ exists as a separated Artin stack with finite diagonal over $\mathbb{Z}[1/30]$ by \cite[Corollary 10.2]{7authors}, but whether it is proper remains still an open question.
	Our main result is the following, where locally-stable reduction means a weakening of semi-stable reduction, see \autoref{def-loc-stable} for the precise definition. 
	
	\begin{theorem}\label{t-main}
	Assume the existence of locally-stable reduction  for surfaces. 
	Then the moduli stack $\overline{\mathcal{M}}_{2,v}$ of stable surfaces of volume $v$ is proper over $\mathbb{Z}[1/30]$.
	\end{theorem}

	The main technical result needed to prove \autoref{t-main} concerns the depth of 3-dimensional log canonical singularities which we briefly explain.
    In \cite[Theorem 6.0.5]{Pos21}, Posva reduced the valutative criterion of properness for the stack $\overline{\mathcal{M}}_{2,v}$ to two conjectures on families of stable surfaces over a DVR: roughly speaking,
    the existence of semi-stable reduction and the $(S_2)$-condition of the central fibre of a locally stable family of surfaces. 
    To prove the $(S_2)$-condition of the central fibre, it is thus natural to study the $(S_3)$-condition at a closed point $x$ of a log canonical 3-fold singularity.
	
	Let us first explain what the tools used to prove the $(S_2)$-condition in characteristic 0 are, as we will mimic this approach.
	In \cite[Theorem 7.20]{kk-singbook}, a local version of the Kawamata--Viehweg vanishing theorem concerning the depth of divisorial sheaves on divisorially log terminal (dlt) and log canonical pairs in characteristic 0 is presented (similar results were obtained previously by Alexeev and Hacon \cite{Ale08, AH12}).
	
	\begin{theorem} [Local Kawamata--Viehweg vanishing for log canonical pairs]\label{t-dpt-C} 
	Let $(X,\Delta)$ be a log canonical pair over a field of characteristic 0.
	Let $D$ be a $\Z$-divisor such that $D \sim_{\Q} \Delta'$ where $0 \leq \Delta' \leq \Delta$.
	If $x$ is a point which is not the generic point of an lc centre, then $$\depth_{x} \mathcal{O}_X(-D) \geq \min \left\{3, \cod_X x \right\}.$$ 	
	\end{theorem}

    This local vanishing is one of the crucial ingredients for the properness of the moduli functor as shown in \cite[Definition-Theorem 2.3]{k-moduli}, where \autoref{t-dpt-C}
	is used to prove the $(S_2)$-condition on the central fibre of a locally stable family.
	For this reason it is natural to consider whether \autoref{t-dpt-C} remains true in positive and in mixed characteristic.
	The examples of klt not CM 3-fold singularities (see \cite{CT19,Ber21, ABL20}) show that \autoref{t-dpt-C} is false in equicharacteristic $p \leq 5$.
	On the contrary, in \cite{ABL20} the first two authors showed together with Lacini that 3-fold klt singularities are Cohen--Macaulay in characteristic $p>5$ and this was later extended by the second author and Koll\'{a}r in \cite[Theorem 17]{BK20} to a local Kawamata--Viehweg vanishing on 3-dimensional excellent dlt singularities whose residue field is perfect of characteristic $p>5$ (analogue to \cite[Theorem 7.31]{kk-singbook}). 
    Moreover, in \cite[Theorem 3.8]{PS14} the third author and Schwede prove a local Kawamata--Viehweg vanishing for sharply $F$-pure singularities.
	From all these results, it would be natural to expect an analog of \autoref{t-dpt-C} for 3-dimensional log canonical singularities to hold, at least in large characteristic. Unfortunately, we show that this is not the case.
    
    \begin{theorem}[See \autoref{s-counterexample}]\label{t-counterexample}
	    For every prime $p>0$, there exist a 3-dimensional log canonical singularity $x \in X$ such that
	    \begin{enumerate}
	        \item the residue field of the closed point $x$ is perfect of characteristic $p$;
	        \item  $x$ is not a minimal log canonical centre;
	        \item $\depth (\mathcal{O}_{X,x})=2$.
	    \end{enumerate}
    \end{theorem}

Nevertheless, we are able to obtain a weaker local Kawamata--Viehweg vanishing statement, which is sufficient to deduce the properness of the moduli space $\overline{\mathcal{M}}_{2,v}$. See \autoref{sec:basic_notation} for the notion of pair used in the article.

\begin{theorem}\label{t-s3-intro}
	Let $C \subset (X,\Delta)$ be a 1-dimensional minimal log canonical centre of a 3-dimensional log canonical pair $(X,X_0+\Delta)$ and let $x \in C$ be a closed point with perfect residue field of characteristic $p \neq 2, 3$ and $5$. 
	If $X_0$ is Cartier and $x \in X_0$, then $\mathcal{O}_{X,x}$ is $(S_3)$ and $X_0$ is $(S_2)$ at $x$.
\end{theorem}
 
To prove \autoref{t-s3-intro}, we find a clear geometric reason for the failure of the $(S_3)$-condition at a closed point $x$ of a log canonical 3-dimensional singularity $X$ which is not a log canonical centre. 
	
	\begin{theorem} \label{t-char-fail-local-KVV-p}
		Let $(X,\Delta)$ be a 3-dimensional log canonical pair on the spectrum of a local ring, such that the residue field of the closed point $x$ is perfect of characteristic $p \neq 2,3, 5$. 
		Let $C \subset X$ be a 1-dimensional  minimal log canonical centre for $(X,\Delta)$.
		Then there exists a proper birational modification $f \colon Z \to X$ such that
		\begin{enumerate}
			\item $Z$ is Cohen-Macaulay,
			\item the exceptional divisor $E$ is $(S_2)$ and for each point $c \in C$, $E$ is normal at the generic points of the fibre $E_c$;
			\item $H^2_x(X, \MO_X) \simeq H^0_x(C, R^1f_*\mathcal{O}_E)$.
		\end{enumerate}
		In particular, if $\MO_{X,x}$ is not $(S_3)$, then $E \to C$ has a wild fibre over $x$.
	\end{theorem}

    We now give an overview of the article.
    In \autoref{s-prelims}, we collect the various technical results on surfaces and 3-folds that we need for our proofs.
    In \autoref{s-depth-3fold}, using the MMP for 3-folds and the Kawamata--Viehweg vanishing for log canonical surfaces admitting a morphism to a curve, we show \autoref{t-char-fail-local-KVV-p}.
    In \autoref{torsion}, we review the theory of wild fibers of a fibred surface $f\colon E\to C$, developed by Raynaud in \cite{Ray70}, which we apply in combination with \autoref{t-char-fail-local-KVV-p} to conclude \autoref{t-s3-intro} in \autoref{S3-cond-lc}.
    In \autoref{ss-properness}, we combine the previous results to show \autoref{t-main} and we also present an application to the asymptotic invariance of plurigenera for minimal models of log canonical surfaces of log general type.
    In \autoref{s-counterexample}, we show the counterexample \autoref{t-counterexample} by applying a relative cone construction to an elliptic surface fibration with a wild fibre.

   \begin{remark}
	While completing this work, the first author has found an alternative construction \cite[Theorem 2]{Arv23} that can replace the role played by \autoref{t-char-fail-local-KVV-p} in this work. 
	This construction simplifies some of the technical arguments in this article. 
	Indeed, the main technical difficulty in the present approach is that in \autoref{t-char-fail-local-KVV-p}, the modification $Z$ together with the crepant bounday is not in general dlt, but only \emph{\'etale dlt}. In \cite[Theorem 2]{Arv23} the first author proves an analogous statement using a possibly non-$\mathbb{Q}$-factorial dlt modification. 
	We believe that the various vanishing statements discussed here may be of independent interest.
   \end{remark}

\textbf{Acknowledgments.}
    The authors thank J. Baudin, S.~Filipazzi, C.D.~Hacon, J.~Koll\'{a}r and Q.~Posva for interesting discussions and useful comments on the topic of this article. 
    We are grateful to the reviewer for carefully reading our article and for suggesting multiple improvements. 
    EA is supported by SNF \#P500PT 203083, FB is partially supported by the NSF under grant number DMS-1801851 and the grant \#200021/192035 from the Swiss National Science Foundation, ZsP is supported by the grant \#200021/192035 from the Swiss National Science Foundation and the ERC starting grant \#804334.
\section{Preliminaries}\label{s-prelims}
	
\subsection{Basic notation}
\label{sec:basic_notation}

\begin{notation}
\label{notation:basic}
\emph{Throughout this article we  work over a fixed base ring $R$ and $X$, $Y$ and $Z$  always denote quasi-projective schemes of pure dimension $n$ over $R$, unless otherwise stated.} 

The base ring $R$ is always be assumed to be Noetherian, excellent,  of finite Krull dimension, and admitting a dualising complex $\omega_{R}^{\bullet}$. Furthermore, $R$ will always be assumed to be of pure dimension $d$. Here, and in general in the present article, dimension means  the absolute dimension, not the relative dimension over $R$.
\end{notation}

We normalise $\omega_{R}^{\bullet}$ as in \cite[Section 2.1]{7authors}: $H^{-i}(\omega_R^{\bullet}) =0$ if $i>d$ and with  $H^{-d}(\omega_R^{\bullet}) \neq 0$. The first non-zero cohomology sheaf $\omega_R \coloneqq H^{-d}(\omega_R^{\bullet})$ is the dualising sheaf of $R$. 
For the upper-shriek functor, we follow the convention of  \cite[\href{https://stacks.math.columbia.edu/tag/0A9Y}{Tag 0A9Y}]{stacks-project}.
By \cite[\href{https://stacks.math.columbia.edu/tag/0AA3}{Tag 0AA3}]{stacks-project}, the complex $\omega_X^{\bullet} \coloneqq \pi^{!} \omega_R^{\bullet}$ is a dualising complex for $X$, where $\pi \colon X \to \Spec(R)$ is the structure morphism.
We then define the dualising sheaf  $\omega_X$ of $X$ to be the first non-zero cohomology sheaf of the complex $\omega_X^\bullet$.

We say that $X$ is a curve (resp.~ a surface, a 3-fold) if it is a connected reduced scheme of dimension 1 (resp.~ 2, 3). We say a proper morphism $f \colon X \to Y$ is a \emph{contraction} if $f_*\mathcal{O}_X=\mathcal{O}_Y$.

Given a closed subscheme $Z$ of $X$, we denote by $\Gamma_{Z,X}$ the functor of global sections with support on $Z$. The induced right derived functor is denoted by $R\Gamma_{Z,X}$ and its $i$-th cohomology by $R^i \Gamma_{Z,X}$ (or $H^i_Z(X, -)$). These groups are called the $i$-th local cohomology groups with support on $Z$.

A Weil $\bQ$-divisor $D$ on a connected reduced scheme $X$ is a formal sum of codimension $1$ integral subschemes with rational coefficients. 
As we will work with non-normal schemes, we recall the definition of the more restrictive class of Mumford divisors following \cite{k-moduli}.
    \begin{definition}
		A Weil $\bQ$-divisor $B$ on $X$ is called a \emph{Mumford $\bQ$-divisor} if $X$ is regular at all
		generic points of $\Supp B$.
	\end{definition}
Equivalently, $B$ is Mumford if $\Supp B$ does not contain any irreducible component of codimension 1 of the divisorial part of the conductor $D \subset X$.
We say a Mumford $\mathbb{Q}$-divisor $D$ is $\mathbb{Q}$-Cartier if there exists $n>0$ such that $nD$ is a Cartier divisor. 
We refer to \cite[Section 2.5]{7authors} for the various notions of positivity (as ample, nef, big) for $\mathbb{Q}$-Cartier $\mathbb{Q}$-divisors. 
When $R=k$ is a field and $X$ is integral, then for a nef $\mathbb{Q}$-Cartier $\mathbb{Q}$-divisor $L$ the \emph{numerical dimension} of $L$ is $\nu(L):=\text{max} \left\{n \geq 0 \mid L^{n} \neq 0 \right\}$.
\begin{definition}
We say $(X, \Delta)$ is a \emph{couple} if 
\begin{enumerate}
    \item $X$ is a reduced, pure-dimensional, $(S_2)$ and $(G_1)$ (where the latter means that  $\omega_X$ is locally free at codimension 1 points of $X$) scheme,
    \item $\Delta$ is an effective Mumford $\mathbb{Q}$-divisor.
\end{enumerate} 
\end{definition}
An open set $U \subset X$ is \emph{big} if $\codim_{X}(X \setminus U) \geq 2$.
If $X$ is $(S_2)$ and $(G_1)$, then any reflexive sheaf is determined on big open sets  and a Mumford divisor $D$ defines a reflexive sheaf $\mathcal{O}_X(D)$ \cite{Hartshorne_Generalized_divisors_on_Gorenstein_schemes}. 
By \cite[Paragraph 5.6]{kk-singbook}, in this case the Mumford class group coincide with the group of isomorphism classes of reflexive sheaves of rank 1 that are locally free in codimension $1$. As a consequence,  if $X$ is pure-dimensional, $(S_2)$ and $(G_1)$, then there exists a Mumford divisor $K_X$ such that $\mathcal{O}_X(K_X) \simeq \omega_X$ (note $\omega_X$ is reflexive by \cite[\href{https://stacks.math.columbia.edu/tag/0AWN}{Tag 0AWN}]{stacks-project}). 
\begin{definition}
We say a couple $(X, \Delta)$ is a \emph{pair} if  $K_X+\Delta$ is a Mumford $\mathbb{Q}$-Cartier divisor. 
\end{definition}
If $(X, \Delta)$ is a pair and $X$ is normal, for every proper birational morphism of normal schemes $\pi \colon Y \to X$, we can write 
$$K_Y+\pi_*^{-1}\Delta=\pi^{*}(K_X+\Delta)+\sum_i a(E_i, X, \Delta) E_i,$$
where $E_i$ run through the $\pi$-exceptional divisors and $a(E_i, X, \Delta) \in \mathbb{Q}$ are called the \emph{discrepancies} of $E_i$ with respect to $(X,\Delta)$. We define $\Delta_Y:=\pi_*^{-1}\Delta -\sum_i a(E_i, X, \Delta) E_i$ to be the \emph{crepant pull-back} of $\Delta$ on $Y$.
We say that $(X, \Delta)$ is a \emph{klt} (resp.~ log canonical or \emph{lc}) pair if $X$ is normal and for every proper birational maps of normal schemes $\pi \colon Y \to X$ then $\lfloor{\Delta_Y \rfloor} \leq 0$ (resp.~ the coefficients of $\lfloor{\Delta_Y \rfloor}$ are $\leq 1$). 

We say a pair $(X,D)$ of pure dimension $n$ is \emph{snc} (=simple normal crossings) if for every closed point $x \in X$, $X$ is regular at $x$ and there exists local coordinates $t_1, \dots, t_n$ such that $\Supp(D) \subset (t_1 \cdots t_n=0)$. 
Note that being snc is a local property in the Zariski topology (but not in the \'etale topology). 
We denote by $\nsnc(X, \Delta)$ the non-snc locus of $(X, \Delta)$.
We say $(X, \Delta)$ is \emph{dlt} if it is log canonical and for every exceptional divisor $E$ such that $\cent_X(E) \subset \nsnc(X, \Delta)$ we have $a(E_i, X, \Delta)>-1.$ 

\begin{definition}
Let $(X, \Delta)$ be a pair. The \emph{\'etale-snc} locus $\etsnc(X, \Delta)$ is the locus where $(X, \Delta)$ is snc in the \'etale topology. This is a Zariski open set of $X$. 

We say that a pair $(X,\Delta)$ is \emph{\'etale-dlt}  if for every exceptional divisor $E$ over $X$ such that $\cent_X E \subseteq X \setminus \etsnc(X, \Delta)$ we have $a(E; X, \Delta) > -1$. 
\end{definition}

\subsection{Semi-log canonical singularities}

If $X$ is reduced, then we can consider its normalisation morphism $\pi \colon \overline{X} \to X$ (see \cite[\href{https://stacks.math.columbia.edu/tag/035N}{Tag 035N}]{stacks-project}). The \emph{conductor} ideal of $\pi$ is the largest ideal $\sI$ of $\sO_X$ which is also an ideal of $\sO_{\oX}$. It can be also defined explicitly in multiple ways:
\begin{multline*}
\sI=\bigset{s \in \sO_X}{s \cdot \pi_* \sO_{\overline{X}} \subseteq \sO_X \textrm{ as subsheaves of the field of total fractions}} 
\\ \im \Big(\bigset{\phi \in \Hom_{\pi_* \sO_{\oX}}(\pi_* \sO_{\oX}, \pi_* \sO_{\oX})}{\im \phi \subseteq \sO_X} \to \sO_X \Big) 
\\ = \im \big( \mathcal{H}\text{om}_{\mathcal{O}_X}(\pi_{*}\mathcal{O}_{\overline{X}},\mathcal{O}_X) \to \sO_X \big)
\end{multline*}
The \emph{conductor subscheme} $D$ of $X$ (resp.~ $\overline{D}$ of $\oX$) is the subscheme defined by $\mathcal{I}$ in $X$ (resp.~ in $\overline{X}$).

We recall the definition of the singularities of the MMP for non-normal varieties, following \cite{kk-singbook}.
We start by explaining what a node is.

\begin{definition}[{\cite[1.41]{kk-singbook}}]\label{def: node}
We say that a scheme $S$ has a \emph{node} at a codimension 1 point $s \in S$ if $\mathcal{O}_{S,s} \simeq A/(f)$, where $(A, \mathfrak{m})$ is a regular local ring of dimension 2, $f \in \mathfrak{m}^2$ and $f$ is not a square in $\mathfrak{m}^2\setminus \mathfrak{m}^3$. Sometimes we equivalently say that $S$ is nodal at $s \in S$.
\end{definition}

\begin{remark}
    Let $(A, \mathfrak{m})$ be a regular local ring of dimension 2 and let $f \in \mathfrak{m}^2$ such that $\Spec(A/(f))$ is a node.  
    It is easy to see that the effective divisor given by $(f=0)$ has multiplicity 2 at the closed point of $\Spec(A)$ and $f$ is not irreducible if and only if the pair $(W, D):=(\Spec(A), (f=0))$ is snc. 
    Examples where the pair is not snc are given in \cite[Examples, page 1]{kk-singbook}.
    If $\text{char}(k(s)) \neq 2$, then it is clear that there exists an \'etale neighborhood $V$ of $W$ for which $(V,D_V)$ is snc.
\end{remark}

\begin{definition}
The scheme $X$ (quasi-projective over $R$ as assumed in \autoref{notation:basic}) is said to be \emph{demi-normal} if it is pure-dimensional, it satisfies Serre's condition $(S_2)$ and its codimension 1 points are either regular or nodal. 
\end{definition}

If $X$ is demi-normal, then $D$ and $\oD$ are reduced closed subschemes of pure codimension $1$ (see \cite[Line 14 of page 189]{kk-singbook}.
We use the following definition of semi-log canonical pairs  in the present article.

\begin{definition}
\label{def:slc}
We say $(X, \Delta)$ is a \emph{semi-log canonical pair} (or \emph{slc}) if 
\begin{enumerate}
    \item $X$ is demi-normal and $(X, \Delta)$ is a pair;
    \item the normalised pair $\big(\overline{X}, \overline{D}+\overline{\Delta}\big)$ is log canonical, where $\overline{D}$ is the conductor subscheme.
\end{enumerate}
\end{definition}

Note that in \autoref{def:slc}, $\big(\overline{X}, \overline{D}+\overline{\Delta}\big)$ is automatically a pair as it is crepant to $(X, \Delta)$. 
As in this article we are interested in understanding the locus of strictly log canonical  singularities, we recall the terminology on log canonical places and centres.

\begin{definition}
	 Let $(X,\Delta)$ be a pair. We denote by $\nklt(X,\Delta)$ the \emph{non-klt locus} of $(X,\Delta)$, which is the closed subset of $X$ consisting of points $x$ of $X$ for which $(X,\Delta)$ is not klt near $x$. 
	 
	 Let $(X,\Delta)$ be a log canonical pair. We say that an irreducible exceptional divisor $E$ for proper birational modification $f \colon Y\to X$ is an \emph{log canonical place} if $a(E,X,\Delta)=-1$.
	 A closed subset $Z \subset X$ is a \emph{log canonical centre} if there exists a log canonical place $E$ such that $\cent_X(E)=Z$.
\end{definition}

We recall the construction of double covers of demi-normal varieties explained in \cite[5.23]{kk-singbook}. This allows to reduce many questions to slc pairs whose irreducible components are regular in codimension 1.

\begin{proposition}[{\cite[5.23]{kk-singbook}}] \label{l-2cover-slc}
Let $(X, \Delta)$ be a slc pair such that $\frac{1}{2} \in \mathcal{O}_X$. Then there exists a finite morphism $\pi \colon \widetilde{X} \to X$ of degree 2 such that
\begin{enumerate}
    \item $\widetilde{X}$ is $(S_2)$, 
    \item $\pi$ is \'etale in codimension 1,
    \item the irreducible components of $\widetilde{X}$ are $(R_1)$ (\emph{i.e.} regular in codimension 1);
    \item the normalisation of $\widetilde{X}$ is a disjoint union of two copies of the normalisation of $X$,
    \item if $K_{\widetilde{X}} + \Delta_{\widetilde{X}} = \pi^*(K_X+\Delta)$, the pair $(\widetilde{X}, \Delta_{\widetilde{X}})$ is slc.
\end{enumerate}
\end{proposition}

	We need also a non-normal version of dlt-ness.
	We refer to \cite[Definition 1.10]{kk-singbook} for the definition of semi-snc pair.

\begin{definition}
An slc pair $(X,\Delta)$ is \emph{semi-dlt} if $a(E,X,\Delta)>-1$ for every exceptional divisor $E$ such that the generic point of $\cent_X E$ is contained in the locus where $(X,\Delta)$ is not semi-snc, where semi-snc is defined in \cite[Def 1.10]{kk-singbook}.
\end{definition}

As for dlt, the notion of semi-dlt is not local in the \'etale topology.

\subsection{Log canonical surface singularities}

	In this section we collect some results on 2-dimensional excellent surface singularities, relying on the classification scheme of \cite[Section 3.3]{kk-singbook}.
	
	\begin{notation}
	Besides the assumptions on our base ring $R$ stated in \autoref{notation:basic}, in the present section we suppose $R$ is integrally closed, local and of dimension 2 with maximal ideal $\mathfrak{m}$ and residue field $k:=R/\mathfrak{m}$. 
	Additionally we set $X= \Spec R$, and we set $\Delta$ to be a $\bQ$-divisor on $X$ for which $(X, \Delta)=(\Spec(R), \Delta)$ is a pair. 
	We denote by $x \in X$ the closed point of $X$. 
	\end{notation}
	
\begin{definition}\label{def-mnm-log}
	Let $\pi \colon Y \to X$ be a projective birational morphism of normal surfaces. We say $\pi$ is a \emph{log minimal resolution} of $(X, \Delta)$ if
	\begin{enumerate}
		\item $Y$ is regular and $\pi_*^{-1} \lfloor{ \Delta \rfloor}$ is regular (as a closed subscheme);
		\item $K_Y+\pi_*^{-1} \Delta$ is $\pi$-nef;
		\item $\mult_y \pi_*^{-1} \Delta \leq 1$ for every $y \in Y$;
		\item the support of $\Ex(\pi) + \pi_*^{-1}\Delta$ has a node at every intersection point of  $\Ex(\pi)$ and $ \pi_*^{-1} \lfloor \Delta \rfloor$.
	\end{enumerate}
\end{definition}

\begin{remark} \label{rem-construction-log-mininal-resolution}
    The existence of a log minimal resolution for surfaces is proven in \cite[Theorem 2.25.a]{kk-singbook}. 
    The construction goes as follows: if $f \colon W \to X$ is a projective log resolution of $(X, \Delta)$ such that $f_*^{-1}\Delta$ is regular, then $Y$ is obtained as the output of a $(K_W+f_*^{-1}\Delta)$-MMP over $X$.
\end{remark}
	
	We need a slightly modified version of the above in the case of dlt surfaces.
	\begin{lemma} \label{l-nice-resolution-dlt}
	  Assume that $(X, \Delta)$ is a dlt surface pair. Then there exists a projective birational morphism $\pi \colon Y \to X$ such that
	    \begin{enumerate}
	        \item $Y$ is regular and $K_Y+\pi_*^{-1} \Delta$ is $\pi$-nef;
	        \item if $K_Y+\Gamma \sim_{\mathbb{Q}} \pi^*(K_X+\Delta)$, then $\lfloor \Gamma \rfloor = \pi_*^{-1} \lfloor \Delta \rfloor$;
	        \item the support of $\Ex(\pi) + \pi_*^{-1}\Delta$ has a node at every intersection point of  $\Ex(\pi)$ and $ \pi_*^{-1} \lfloor \Delta \rfloor$.
	    \end{enumerate}   
	 \end{lemma}
	 
	 \begin{proof}
	   Let $f \colon W \to X $ be a thrifty log resolution of $(X,\Delta)$ (\cite[Lemma 2.79]{kk-singbook}). As $X$ is a surface,  $f$ being thrifty means that it is an isomorphism at the nodes of  $\lfloor \Delta \rfloor$. This we can achieve by running our resolution algorithm by excluding the nodes of  $\lfloor \Delta \rfloor$. 
	   
	   Run then a $(K_W + f_*^{-1}\Delta)$-MMP over $X$ ending with $\pi \colon Y \to X$ such that $K_Y+\pi_*^{-1}\Delta$ is $\pi$-nef.
	   As this is also a $K_W$-MMP we deduce $Y$ is regular by \cite[Theorem 2.29]{kk-singbook}. 
	   As $f$ does not extract log canonical places, (b) is immediate.
	   To finally verify (c) we argue as in the proof of \cite[Theorem 2.25.a]{kk-singbook}.
	 \end{proof}

Our next goal is a precise understanding of the exceptional divisor of a log minimal resolution of a log canonical  singularity.
We start by recalling \cite[Theorem 2.31]{kk-singbook} on the reduced boundary of 2-dimensional log canonical  singularities.

	\begin{theorem}\label{nodes} Assume that $(X,\Delta= E + D)$ is log canonical, where $E=\sum_i E_i$ has only coefficients 1. Then either:
		\begin{enumerate}
			\item $E$ is regular at $x$, or
			\item $E$ has a node at $x$, no components of the support of $D$ contain $x$ and every exceptional divisor of a minimal log resolution has discrepancy $-1$.
		\end{enumerate} 
	\end{theorem}

We will need the following observation on conics.

	\begin{lemma}\label{l-nodal}
		Let $k$ be a separably closed field and let $C$ be a $k$-projective integral Gorenstein curve.
		Suppose that $C$ is singular and $\characteristic k \neq 2$. 
		Then $\deg_k \omega_{C/k} \geq 0$.
	\end{lemma}

\begin{proof}
Without loss of generality we can suppose $k=H^0(C, \mathcal{O}_C)$.
Suppose by contradiction that $\deg_k \omega_{C/k} <0$.
By \cite[Lemma 10.6]{kk-singbook} $C$ embeds as a conic in $\mathbb{P}^2_k$. 
Taking the base change to $\overline{k}$, we still get an embedding $C_{\overline{k}} \rightarrow \mathbb{P}^2_{\overline{k}}$.
As $k$ is separably closed, $C_{\overline{k}}$ is still an irreducible conic and by the classification of conics over an algebraically closed field, either $C_{\overline{k}}$ is regular or it is a double line.
Note the case of double line cannot appear as $\characteristic(k) \neq 2$ by
\cite[Lemma 2.17]{BT22}.
Finally if $C_{\overline{k}}$ is regular, we deduce $C$ is regular by descent for faithfully flat morphisms \cite[\href{https://stacks.math.columbia.edu/tag/033E}{Tag 033E}]{stacks-project}, getting a contradiction.
\end{proof}

\begin{example}
	The following examples shows that the assumptions in \autoref{l-nodal} are sharp.
	Let $k$ be a field and consider the conic 
	$$C:= \left\{ x^2-uy^2=0 \right\} \subset \mathbb{P}^2_{k}=\Proj k[x,y,z],$$
	where $u \in k$.
 	By the Jacobian criterion \cite[\href{https://stacks.math.columbia.edu/tag/07PF}{Tag 07PF}]{stacks-project}), it is easy to see that the only non-regular point of $C$ is $p=[0:0:1]$. Note that $\deg_k \omega_{C/k} =-2$. This example shows that the assumptions of \autoref{l-nodal} are indeed necessary:
 \begin{itemize}
     \item If $\characteristic  k \neq 2$ and $u \not\in k^2$, then $C$ is integral, singular and with $\deg_k \omega_{C/k} <0$ but $k$ is not separably closed. 
     \item If $\characteristic  k \neq 2$ and $k$ is separably closed, then $C$ is  singular and with $\deg_k \omega_{C/k} <0$ but $C$ is not integral.
       \item If $\characteristic  k = 2$, $k$ is separably closed and $u \not\in k^2$, then $C$ is  integral, singular and with $\deg_k \omega_{C/k} <0$. Geometrically, $C_{\overline{k}}$ is a double line.
 \end{itemize}
 \end{example}

\begin{proposition}\label{p-lmn-etale-dlt}
Assume that $(R,m)$ is strictly Henselian with $ \characteristic k=p \neq 2$, and that $(X, \Delta)$ is log canonical  such that $\Delta$ is a $\mathbb{Q}$-divisor.
	Let $\pi \colon Y \to (X, \Delta)$ be a log minimal resolution as in \autoref{def-mnm-log}.
	Then one of the following holds:
	\begin{enumerate}
	    \item \label{itm:p-lmn-etale-dlt:one_exceptional} $\Delta=0$ and there exists an irreducible nodal curve $E \subset \Ex(\pi)$. Then $\Ex(\pi)=E$, $K_Y + E \equiv_X 0$, and $(Y,E)$ is \'etale-snc;
	    \item \label{itm:p-lmn-etale-dlt:general} $(Y,\pi_*^{-1} \lfloor{\Delta \rfloor}+\Ex(\pi))$ is snc. 
	\end{enumerate}
\end{proposition}	

\begin{proof} 
The proof is case by case.

    \ptofpf{There exists a singular exceptional  curve $E \subset \Ex(\pi)$:} 
    By \autoref{l-nodal}, $\deg_k \omega_{C/k} \geq 0$ in this case. By \cite[3.30.1]{kk-singbook} then $\Delta=0$ and $E$ is the unique exceptional divisor. Let $-1 \leq a \in \bQ$ be the discrepancy of $E$. By adjunction we have 
    \begin{equation}
    \label{eq:p-lmn-etale-dlt:E}
        0 \leq \deg K_E \leq (K_X +E) \cdot E = (K_X - aE) \cdot E + (1+a) E^2    = (1+ a ) E^2 \leq 0.
    \end{equation} 
    In particular we have equality everywhere. Taking into account that $E^2<0$, this means that $a=-1$. We obtain that $(Y,E)$ is log canonical. Taking into account that  $Y$ is regular, we deduce that $(Y,E)$ is \'etale-snc. To see $K_Y + E \equiv_X 0$, we just note that $(K_Y +E ) \cdot E =0$ by \autoref{eq:p-lmn-etale-dlt:E}.
    
 \ptofpf{ All irreducible component $E_i$ of $\Ex(\pi)$ are regular:} Note that $\pi_*^{-1} \lfloor{\Delta \rfloor}$ is regular and that $\Ex(\pi) + \pi_*^{-1} \lfloor{\Delta \rfloor}$ is snc at intersection points by construction of the minimal resolution, in this case.
    We are left to prove $\Ex(\pi)$ is snc.
    We can suppose all irreducible components are conics by \cite[3.30.2]{kk-singbook} and we fix $r_i:=\dim_k H^0(E_i, \mathcal{O}_{E_i})$.
    
    \ptofpf{All the $E_i$ are regular, and $(E_i \cdot E_j) > \max\left\{r_i, r_j\right\}$ for some $i \neq j$:} in this case $\Delta=0$ and  there are two exceptional curves $E_1$ and $E_2$, by \cite[3.30.3]{kk-singbook}. Consider the following computation
    \begin{multline*}
        0 \leq E_1 \cdot E_2  + \deg K_{E_1}  \leq (K_X +E_1 + E_2) \cdot E_1 
        \\ = (K_X - a_1E_1 - a_2E_1) \cdot E_1 + (1+a_1) E_1^2 + (1+a_2) E_1 \cdot E_2 = (1+a_1) E_1^2 + (1+a_2) E_1  \cdot E_2
    \end{multline*} 
    This implies that 
    \begin{equation}
    \label{eq:p-lmn-etale-dlt:E_1}
        -(1+a_1) E_1^2 \leq (1+ a_2) E_1 \cdot E_2
    \end{equation}
    By applying the same argument to $E_2$ instead of $E_1$ we obtain
    \begin{equation}
    \label{eq:p-lmn-etale-dlt:E_2}
        -(1+a_2) E_2^2 \leq (1+ a_1) E_1 \cdot E_2
    \end{equation}
    Multiplying \autoref{eq:p-lmn-etale-dlt:E_1} and \autoref{eq:p-lmn-etale-dlt:E_2} together we obtain the following, where we are also using  that both sides of the two inequalities are non-negative. 
        \begin{equation}
    \label{eq:p-lmn-etale-dlt:together}
        (1+a_2) (1+a_1) (E_1^2)( E_2^2) \leq (1+ a_1) (1+a_2) (E_1 \cdot E_2)^2
    \end{equation}
    With other words, either one of the $a_i$ is equal to $-1$ or the determinant of the intersection matrix is non-positive. The latter contradicts the negative definiteness of the intersection matrix, hence we obtain that one of the $a_i$ is $-1$. By symmetry we can assume that $a_1=-1$. However, then \autoref{eq:p-lmn-etale-dlt:E_2} says that $(1+a_2) E_2^2 \geq 0$. As $E_2^2<0$ this implies that also $a_2=-1$. 
    
    In particular, $(Y, E_1 + E_2)$ are log canonical and hence by adjunction so is $(E_1, E_1 \cap E_2)$ and $(E_2, E_1 \cap E_2)$. This means that the coefficients of $E_1 \cap E_2$ are $1$ on both $E_1$ and $E_2$. As $k$ is separably closed, all finite non-trivial extensions of $k$ have degree divisible by $p$. As $E_1 \cdot E_2 =2$ and $p \neq 2$ this means that in fact $E_1 \cap E_2$ contains only points with residue field equal to $k$. As above we have seen that the coefficients of these points cannot be more than $1$, we obtain that $E_1 \cap E_2$ has two distinct points with coefficient $1$ and hence the intersection of $E_1$ and $E_2$ is transversal. In particular $(Y, E_1 + E_2)$ is snc (and the singularity is a cusp with the exceptional divisor of the minimal resolution being a cycle of two conics).
    
    \ptofpf{All the $E_i$ are regular, and $(E_i \cdot E_j) =\max \left\{r_i, r_j\right\}$ for all $i$ and $j$:} fix two components $E_i$ and $E_j$. We may assume by symmetry that $r_i \geq r_j$. In particular the intersection scheme $E_i \cap E_j$ is a length one Artinian scheme over $H^0(E_i, \sO_{E_i})$. This implies that $E_i \cap E_j$ is reduced. Hence the intersection of $E_i$ and $E_j$ is transversal, which concludes our proof. 
\end{proof}

The following is well-known.
\begin{lemma} \label{l-pushfor-nef}
		Let $f \colon Y \to Z $ be a projective birational morphism of normal surfaces over $R$, and let $D$ be a nef $\Q$-Cartier divisor on $Y$. 
		If $f_*D$ is $\mathbb{Q}$-Cartier, then it is nef. 
\end{lemma}
\begin{proof}
	Let $C$ be a curve on $Z$, mapping to a closed point of $\Spec(R)$. 
	By projection formula for the Mumford pull-back we conclude $f_*D \cdot C =D \cdot f^*C \geq 0$. 
\end{proof}

\begin{corollary}
   Assume that the  characteristic of $k$ is  $p \neq 2$, and that    $(X=\Spec(R), \Delta)$ is log canonical.
    Then there exists a projective birational morphism $f \colon Z \to X$ such that 
    \begin{enumerate}
        \item $\big(Z,f_*^{-1}\lfloor \Delta \rfloor+\Ex(f)\big)$ is \'etale-dlt,
        \item $K_Z+f_*^{-1}\Delta+\Ex(f)=f^*(K_X+\Delta)$, and
        \item $-\Ex(f)$ is nef over $X$.
    \end{enumerate}
\end{corollary}

\begin{proof}
    First, we may assume that $R$ is strictly Henselian. 
    Second, let $\pi \colon Y \to X$ be as in \autoref{def-mnm-log}.  By \autoref{p-lmn-etale-dlt}, there are two cases.
    In case \autoref{itm:p-lmn-etale-dlt:one_exceptional}, $\Delta=0$ and $\big(Y, \Ex(\pi)\big)$ is \'etale-snc and $\Ex(\pi)$ is a single exceptional divisor, necessarily anti-$f$-nef. In this case $f:=\pi$ satisfies the assertion of the theorem.
     
    In case \autoref{itm:p-lmn-etale-dlt:general}, the pair $\big(Y, \pi_*^{-1}\lfloor\Delta\rfloor +\Ex(\pi)\big)$ is snc.
    By \cite[Theorem 1.1.(QF)]{Tan18} we can run a $\big(K_Y+\pi_*^{-1}\Delta+\Ex(\pi) \big) \equiv_X \big(\sum_i (1+a(E_i, X, \Delta) E_i\big)$-MMP over $X$, denoted by $\rho \colon Y \to Z$, ending with a minimal model $f \colon Z \to X$. 
    By a standard application of the negativity lemma \cite{7authors}*{Lemma 2.16}, the birational contraction $\rho$ contracts exactly the $\pi$-exceptional divisors with discrepancy $a(E,X, \Delta)>-1$ and thus $Z$ is a $\mathbb{Q}$-factorial surface with $(K_Z + f_*^{-1}\Delta + \Ex(f)) = f^*(K_X+\Delta)$.
    As a $\big(K_Y+\pi_*^{-1}\Delta+\Ex(\pi)\big)$-MMP over $X$ is a $\big(K_Y+\pi_*^{-1}\lfloor\Delta\rfloor+\Ex(\pi)\big)$-MMP\footnote{This because for any effective divisor $D$ on a surface $X$, the strict transform $\pi_*^{-1}D$ is nef over $X$.}, the pair $\big(Z, f_*^{-1}\lfloor \Delta \rfloor + \Ex(f)\big)$ remains dlt. 
    By the definition of the log minimal resolution, $-\sum_i a(E_i, X, \Delta) E_i$ is nef over $X$ and therefore so is $-\Ex(f)=\rho_*(-\sum_i (1+a(E_i, X, \Delta)) E_i)$ by \autoref{l-pushfor-nef}.
\end{proof}    

\subsection{Dlt modifications and log canonical  centers}

In this section, we recall dlt modifications and apply their existence to the study of log canonical  centres of log canonical 3-folds.
Since we will need the MMP developed in \cite{7authors} we suppose the following:

\begin{notation}
	Besides the assumptions on our base ring $R$ stated in \autoref{notation:basic}, we suppose that the 
characteristic of the residue fields of $R$ are different from $2, 3$ and $5$. 
\end{notation}

\begin{definition}
Let $(X,\Delta)$ be a log canonical pair. A proper birational morphism $\pi \colon (Y,\Delta_Y) \to (X, \Delta)$ is a \emph{dlt modification} (or dlt blow-up) if $(Y, \Delta_Y)$ is dlt, where $K_Y+\Delta_Y \sim_{\mathbb{Q}} \pi^*(K_X+\Delta)$ and $\Delta_Y=f_*^{-1}\Delta+E$, where $E$ denotes the divisorial part of the exceptional locus of $\pi$.
\end{definition}

The existence of a dlt modification for $(X,\Delta)$ extracting only divisors of discrepancies $-1$ is a standard consequence of the MMP (see \cite[Corollary 9.21]{7authors}).

\begin{proposition}\label{dlt-modification}
    Let $(X, \Delta)$ be a log canonical 3-fold pair. 
    Then there exists a dlt modification $Y \to (X, \Delta)$.
\end{proposition}

We recall some properties of log canonical  centres on log canonical  excellent 3-fold pairs.

\begin{proposition}\label{l-inters-lcc}
    	Any intersection of log canonical centres of a 3-dimensional log canonical  pair $(X,\Delta)$ is a union of log canonical centres.
\end{proposition} 

\begin{proof}
    This is \cite[Corollary 1.7]{FW20} (see also \cite[Corollary 5.2.16]{Pos21} for a proof in the case $X$ is defined over $\mathbb{F}_p$).
 \end{proof}

We will need the following characterisation of plt pairs.
	
\begin{corollary}\label{c-minimal-cod1}
	Let $(X, \Delta=D+B)$ be a 3-dimensional log canonical  pair, where $D$ is a prime divisor.
    If $D$ is a minimal log canonical centre, then $(X, \Delta)$ is plt in a neighbourhood of $D$. 
\end{corollary}
	
\begin{proof}
    Suppose by contradiction $(X,\Delta)$ is not plt around $D$.
    By definition, there exists a log canonical centre $S$ such that $Z:=S \cap D$ is not-empty and that $\codim_X Z \geq \codim_X S >1$.
    As $Z$ is a union of log canonical  centres by \autoref{l-inters-lcc}, this contradicts the minimality of $D$.
\end{proof}

We will need the following technical result on dlt singularities. 

\begin{lemma}\label{l-contr-dlt}
    Let $\pi \colon (Y, \Delta_Y) \to (X, \Delta)$ be a proper crepant birational contraction of $\bQ$-factorial pairs. 
    If $(Y, \Delta_Y)$ is  dlt and $a(E, X, \Delta)>-1$ for every $\pi$-exceptional divisor $E$, then $(X,\Delta)$ is dlt as well. 
\end{lemma}
	
\begin{proof}
Let $E$ be a log canonical place over $(X, \Delta)$. 
As $(Y,\Delta_Y)$ is dlt, the generic point of $Z=\cent_Y(E)$ is a stratum  of $(Y, \Delta_Y^{=1})$, and $E$ is already a log canonical place of $(Y, \Delta_Y^{=1})$. This in particular implies that $Z \not\subseteq \Supp \Delta_Y^{<1}$. However, as both $X$ and $Y$ are $\bQ$-factorial, the exceptional locus of $\pi$ is purely divisorial. 
So, putting the last two sentences and the assumption on the discrepancies of the $\pi$-exceptional divisors together we obtain that $Z \not\subseteq \Ex(\pi)$. However that means that $\pi$ is an isomorphism around the generic point of $ \pi(Z) $, and therefore $\pi(Z)$ is also a stratum of $(X, \Delta^{=1})$.
\end{proof}	

Note that the $\mathbb{Q}$-factoriality hypothesis in \autoref{l-contr-dlt} is needed as shown in \cite[Example 3.8.4]{Fuj07-what}.
	
\subsection{A restriction sequence for pairs}

In this section, we refine the short exact sequences used in \cite{HW19, BK20}.
We start by recalling some general properties of codimension 1 strata of dlt pairs.

	\begin{lemma}\label{l-prop-1cod-dlt}
		Let $\big(X,\sum_{i\in I} E_i+\Delta\big)$ be a dlt pair, where $E_i$ are prime divisors and $\lfloor \Delta \rfloor=0$. Then
		\begin{enumerate}
			\item \label{itm:l-prop-1cod-dlt:R_1} $E_i$ is $(R_1)$ (i.e. $E_i$ is regular in codimension 1) for every $i \in I$;
			\item \label{itm:l-prop-1cod-dlt:normalization} the normalisation $n \colon \bigcup E^{n}_i \rightarrow  \bigcup E_i$  is the disjoint union of the $(S_2)$-ifications of the $E_i$'s, and it factorises through the $(S_2)$-ification $\nu \colon E^{\nu} \rightarrow \bigcup E_i$;
			\item \label{itm:l-prop-1cod-dlt:univ_homeom} if $X$ is $\Q$-factorial, $E^{\nu}_{i} \to E_i$ is a universal homeomorphism for every $i \in I$.
		\end{enumerate}  
	\end{lemma}
	
	\begin{proof}
		For \autoref{itm:l-prop-1cod-dlt:R_1}, it is sufficient to localise at codimension 1 points of $E_i$ and apply \cite[Theorem 2.31]{kk-singbook}.
		Then \autoref{itm:l-prop-1cod-dlt:normalization} follows immediately from \autoref{itm:l-prop-1cod-dlt:R_1} and \autoref{itm:l-prop-1cod-dlt:univ_homeom} is proven in \cite[Lemma 2.1]{HW20}.
	\end{proof}

We begin by studying the singularities of the \'etale-dlt surfaces.

\begin{lemma}\label{l-etale-dlt-surfaces}
   Assume that $R$ is local with closed point $x \in X= \Spec R$, and $(X, \Delta= E+D)$ is an \'etale-dlt surface pair such that $\lfloor\Delta \rfloor=E$. Then either
    \begin{enumerate}
        \item \label{itm:l-etale-dlt-surfaces:dlt} $(X,E+D)$ is dlt at $x$; or
        \item \label{itm:l-etale-dlt-surfaces:node} $X$ is regular, $E$ is irreducible with  a node at $x$ and $\Delta=0$.    
    \end{enumerate}
    In particular, $X$ is $\mathbb{Q}$-factorial.
\end{lemma}

\begin{proof}
    If $E=0$, then $\lfloor \Delta \rfloor =0$, and hence $(X, \Delta)$ is klt. This is covered by point \autoref{itm:l-etale-dlt-surfaces:dlt}. Hence we may assume that $E \neq 0$. As we work in the local case, this means that $x \in \Supp E$.
    By \cite[Proposition 2.15]{kk-singbook} $(X,\Delta)$ has log canonical singularities. If there is an irreducible component of $E$ which is singular, then $E$ is an irreducible nodal curve and $\Delta=0$ by \autoref{nodes}. As $(X, \Delta)$ is \'etale-dlt, in this case $X$ is regular, so we are in case \autoref{itm:l-etale-dlt-surfaces:node}.
    Thus we may also assume that every irreducible component of $E$ is regular. We may also assume that $x \not\in \etsnc(X, \Delta)$. Note that $\Spec R \setminus \{x \} \subseteq \snc(X, \Delta)$, and all discrepancies over the point $x$ are greater than $-1$ by the \'etale snc assumption. Hence $(X, \Delta)$ is actually dlt at $x$. 
    
    For the  assertion about $\bQ$-factoriality, we conclude in case \autoref{itm:l-etale-dlt-surfaces:dlt} by combining \cite[Proposition 2.28]{kk-singbook} and \cite[Proposition 10.9]{kk-singbook} and case \autoref{itm:l-etale-dlt-surfaces:node} is immediate.
\end{proof}

For \'etale-dlt surface pairs we need the following statement on the existence of a special resolution not extracting log canonical places.

\begin{lemma}\label{l-res-not-nodes}
    Let $(X, \Delta=E+\Gamma)$ be an \'etale-dlt surface pair such that $\lfloor \Delta \rfloor=E$.
    Then there exists a projective birational morphism $\pi \colon Y \to X$ such that 
    \begin{enumerate}
        \item $Y$ is a regular surface, and
        \item by setting $K_Y+\Delta_Y = \pi^* (K_X+\Delta)$, we have $\Delta_Y \geq 0$ and $\lfloor{ \Delta_Y \rfloor}=\pi_*^{-1}(E)$.
    \end{enumerate}
\end{lemma}

\begin{proof}
    Let $x \in (X,  \Delta)$ be a closed point. We divide the proof in two cases.
    If  $(X,\Delta)$ is dlt near $x$, we take the resolution of singularities at $x$ constructed in \autoref{l-nice-resolution-dlt}.
    If $(X,\Delta)$ is not dlt near $x$, we do not perform any blow-up as $X$ is already regular around $x$ by \autoref{l-etale-dlt-surfaces}.
    \end{proof}

We need the following generalisation of the short exact sequence of \cite[Section 3]{HW19} to \'etale-dlt surface pairs. 
	
\begin{lemma} \label{l-dlt-seq-curves}
	Let $(X, \Delta)$ be a log canonical surface pair.
	Suppose $\Delta = E +\Delta'$, where $E$ is a $\mathbb{Z}$-divisor and $(X,E)$ is \'etale-dlt.
	Let $D$ be a $\mathbb{Z}$-divisor on $X$ such that $\Supp D$ does not contain any irreducible component of $E$ or any point of $\Sing E$.
	Then there exists a canonically defined Mumford $\mathbb{Z}$-divisor $D_E$ on $E$ such that
	\begin{enumerate}
	    \item $D_E \sim_{\mathbb{Q}} D|_E + \Gamma_E$ for some Mumford $\bQ$-divisor $0\leq \Gamma_E \leq \Diff_E(0)\leq \Diff_E(\Delta')$;
		\item  there exists a short exact sequence of $\mathcal{O}_X$-modules
		$$0 \to \omega_X(D) \to \omega_X(E+D) \to \omega_E(D_E) \to 0.$$
	\end{enumerate} 
\end{lemma}

\begin{proof}
		Recall that $X$ is $\mathbb{Q}$-factorial by \autoref{l-etale-dlt-surfaces}.
		By the assumption that  no irreducible component of $ E$ and no point of the singular locus of $E$ is contained in the support of $D$, the divisor $D|_E$ is a well-defined Mumford $\mathbb{Q}$-divisor on $E$.
		
		Let $\pi \colon Y \to X$ be the resolution of the pair $ (X, E)$ given by \autoref{l-res-not-nodes} and write $K_Y+E_Y+\Gamma_Y=\pi^{*}(K_X+E)$, where $\lfloor E_Y + \Gamma_Y \rfloor = E_Y=\pi_*^{-1}(E)$. As $\pi$ extracts no divisor of discrepancy $-1$, $\pi$ is an isomorphism around the singular points of $E$. Hence,  $ \pi|_{E_Y} : E_Y \to E$ is an isomorphism. For similar reasons, $\lceil \pi^* D\rceil|_{E_Y}$ does make sense, i.e., the support of $\lceil \pi^* D\rceil$ intersects $E_Y$ only along its regular locus. Let $D_E$ be the divisor on $E$ corresponding to the divisor $\lceil \pi^* D\rceil|_{E_Y}$ on $E_Y$ via the isomorphism $\pi|_{E_Y}$.
		
		As  $Y$ is regular, we have the following exact sequence on $Y$:
		\begin{equation}
		\label{eq:l-dlt-seq-curves:before_pushforward}
		    0 \to \omega_Y(\lceil \pi^*D \rceil) \to \omega_Y(E_Y+ \lceil \pi^{*}D \rceil) \to \omega_{E_Y}\left(\lceil \pi^{*}D \rceil|_{E_Y}\right) \to 0.  
		\end{equation}
		Note the following properties:
		\begin{enumerate}
			\item since $\pi$ does not extract any divisor of discrepancy $-1$, we have $K_Y+\lceil \pi^{*}D \rceil \geq \lfloor K_Y+\Gamma_Y+\pi^*D\rfloor\geq \lfloor \pi^{*}(K_X+D)\rfloor $ so $\pi_*(\omega_Y(\lceil \pi^*D \rceil))= \omega_{X}(D)$. Similarly $\pi_*(\omega_Y(E_Y+ \lceil \pi^{*}D \rceil))=\omega_X(E+D)$;
			\item By the above choice of $D_E$ we have $\pi_* \omega_{E_Y}\left(\lceil \pi^{*}D \rceil|_{E_Y}\right) \cong \left( \pi|_{E_Y} \right)_* \omega_{E_Y}\left(\lceil \pi^{*}D \rceil|_{E_Y}\right) \cong \omega_E (D_E) $.
			\item by GR vanishing for surfaces (\cite[Theorem 10.4]{kk-singbook}), $R^1\pi_*\omega_Y(\lceil \pi^*D \rceil)=0$.
		\end{enumerate}
		Thus, pushing forward \autoref{eq:l-dlt-seq-curves:before_pushforward} via $\pi$ we obtain the short exact sequence
		$$0 \to \omega_X(D) \to \omega_X(E+D) \to \omega_E(D_E) \to 0.$$
		We are only left to check that $D_E \sim_{\bQ} D|_E+\Gamma_E$ for some $0 \leq \Gamma_E \leq \Diff_E(0)$.
		Note that via the isomorphism $\pi|_{E_Y}$, $\Gamma_E$ identifies with $(\lceil \pi^{*}D \rceil-\pi^*D)|_{E_Y} \geq 0$.
		Let $x$ be a point of $X$ and let $i_x$ be the determinant of the dual graph of the minimal resolution of $X$ at $x$. By possibly restricting  to a neighbourhood of $x$ we have that $i_x D$ is Cartier by \cite[Prop 10.9.(3)]{kk-singbook}. 
		Additionally, by \cite[Corollary 3.45]{kk-singbook}, the following equality holds: $$ \Diff_E(0) =\begin{cases} \left(1-\frac{1}{i_x}\right)x, & \mbox{if } (X,E) \mbox{ is plt near  } x \\ x, & \mbox{if } (X,E) \mbox{ is not plt near } x,
		\end{cases}$$
		Since $\lfloor \Gamma_E \rfloor=0$ and $i_x \Gamma_E$ is integral, we finally conclude that $\Gamma_E \leq \Diff_E(0) \leq \Diff_E(\Delta'). $
	\end{proof}
	
	In higher dimension we deduce the following generalisation of \cite[Lemma 5]{BK20}.

	\begin{proposition}\label{seqlc} 
	Let $(X, \Delta)$ be a log canonical pair.
	Suppose $\Delta =E+\Delta'$, where $E$ is a $\mathbb{Z}$-divisor and $(X,E)$ is an \'etale-dlt pair. 
	Let $\nu \colon E^{\nu} \to E$ be the $(S_2)$-ification of $E$.
	If $D$ is a $\mathbb{Z}$-divisor on $X$, then there is a short exact sequence of $\mathcal{O}_X$-modules:
	$$0 \to \omega_X(D) \to \omega_X(E+D) \to^{r} \nu_*\left(\omega_{E^{\nu}}(D_{E^{\nu}}) \right), $$
	where $D_{E^{\nu}} \sim_{\mathbb{Q}} D|_{E^{\nu}}+\Gamma_E$ is a Mumford divisor on $E^{\nu}$ for some $\bQ$-divisor $0\leq \Gamma_{E^{\nu}} \leq \Diff_{E^{\nu}}(\Delta')$.
	Moreover, $r$ is a surjection at all codimension 1 points in $E$ and, if $\omega_X(D)$ is $S_3$, then $r$ is surjective.
\end{proposition}
	
\begin{proof}
	 By \autoref{l-etale-dlt-surfaces}, at the codimension $2$ singular points of $E$, the $
	 \mathbb{Z}$-divisor $D$ is Cartier. 
	 Hence, up to replacing $D$ by another divisor in its linear equivalence class, we may assume that $D$ does not contain any component of $E$ and it also does not contain any singular point of $E$ that has codimension $2$ in $X$. 
	 By localising at codimension 2 points of $X$ and applying \autoref{l-dlt-seq-curves}, there exists a canonically defined Mumford $\bZ$-divisor $D_E$ on $E$. 
	 As the irreducible components of $E$ are $(R_1)$, by taking the preimage of $D$ in $E^{\nu}$ we obtain a globally defined Mumford $\mathbb{Z}$-divisor $D_{E^\nu}$ on the $(S_2)$ surface $E^\nu$. 

	Consider the natural exact sequence 
	$$0 \to \omega_X(D) \to \omega_X(E+D) \to \mathcal{Q} \to 0,$$
	
	where $\mathcal{Q} $ is a sheaf supported on $E$. 
	Note that $\mathcal{Q}$ is a torsion-free $\mathcal{O}_E$-module of rank 1 by \cite[Corollary 2.61]{kk-singbook} and therefore the $(S_2)$-hull $\mathcal{Q} \to \mathcal{Q}^{(**)}$ is an injection.  
	So, we obtain the exact sequence 
	$$0 \to \omega_X(D) \to \omega_X(E+D) \to \mathcal{Q}^{**}.$$
 	We now claim that $\mathcal{Q}^{(**)}$ is isomorphic to $\nu_*(\omega_{E^{\nu}}(D_{E^{\nu}}))$.
	By construction of the residue map, there is a natural homomorphism $\psi \colon \mathcal{Q} \to \nu_*(\omega_{E^{\nu}}(D_{E^{\nu}}))$. As $\nu_*(\omega_{E^{\nu}}(D_{E^{\nu}}))$ is $S_2$, then there is a natural map $\mathcal{Q}^{(**)} \to \nu_*(\omega_{E^{\nu}}(D_{E^{\nu}}))$.
	As both $\mathcal{O}_X$-modules are $(S_2)$, it is sufficient to show equality at codimension 1 points of $E$, which has been proved in \autoref{l-dlt-seq-curves}. 
	The linear equivalence $D_{E^{\nu}} \sim_{\Q} D|_{E^{\nu}} + \Gamma_E$ and $0 \leq \Gamma_{E^{\nu}} \leq \Diff_{E^{\nu}}(\Delta')$ is a codimension 2 statement and it is a consequence of \autoref{l-dlt-seq-curves}. 
	
	For the last claim, if $\omega_X(D)$ is $(S_3)$, then $\mathcal{Q}$ is $(S_2)$ by \cite[Lemma 2.60]{kk-singbook}, thus concluding $r$ is surjective.
\end{proof}

\subsection{Partial resolutions of demi-normal excellent surfaces}

In this subsection, we fix an excellent base ring $T$ such that $\frac{1}{2} \in \mathcal{O}_T.$
We start by defining the notion of a pinch point for excellent local rings.

\begin{definition}
    Let $(R, \mathfrak{m})$ be a 2-dimensional excellent local ring. 
    We say it is a \emph{pinch point} if there exists a finite \'etale morphism $R \to S$ such that $S \simeq R'/(x^2-zy^2)$, where $R'$ is a 3-dimensional regular local ring and $(x,y,z)$ is a regular system of parameters for $R'$.  
\end{definition}

We recall the definition of semi-regular surfaces.

\begin{definition}[{cf. \cite[Definition 4.2]{KSB88}}] 
A surface $S$ is called \emph{semi-regular} if for every closed point $s \in S$ the local ring $\mathcal{O}_{S,s}$ is either regular, a node (cf. \autoref{def: node}) or a pinch point.
\end{definition}

Motivated by \cite[Theorem 10.56]{kk-singbook}, we introduce the notion of semi-regularity for surface pairs.

\begin{definition}
    Let $(S, H)$ be a pair, where $S=\Spec(R)$ where $R$ is an excellent local ring of dimension 2 and $H$ is a Weil $\mathbb{Z}$-divisor. 
    We say $(S,H)$ is a \emph{semi-regular} pair if $S$ is semi-regular and one of the following holds:
    \begin{enumerate}
        \item \label{item: regular} the pair $(S, H)$ is snc;
        \item \label{item: nodal} $S$ is nodal and there exists an \'etale morphism  $\Spec R/(x^2-uy^2)  \to S$, where $R$ is 3-dimensional regular ring with $u \in R^*$, local parameters $x,y,z$ and $H=(z=0)$; 
        \item \label{item: pinch}$S$ is a pinch point, and there exists an \'etale morphism $\Spec R/(x^2-zy^2) \to S$, where $R$ is a 3-dimensional regular ring with local parameters $x,y,z$ and $H=(x=z=0)$.
    \end{enumerate}
\end{definition}

\begin{remark} \label{rem: blow-up-non-normal-locus-pinch}
As in the characteristic 0 case, the conductor $D \subset S$ of a semi-regular surface is a regular curve.
In case \autoref{item: pinch}, if $f \colon T \to S$ is the blow-up of $S$ along $D$, then the local computations in \cite[Definition 1.43]{kk-singbook} shows the pair $(T, D_T + f^*H)$ is snc.
\end{remark}

\begin{definition}[{cf. \cite[Definitions 4.3, 4.4]{KSB88}}] \label{def: semi-regular-pair}
    Let $(S, H)$ be a demi-normal surface pair. We say that $\pi \colon T \to (S,H)$ is a \emph{semi-regular resolution} if 
    \begin{enumerate}
        \item $\pi$ is a proper morphism;
        \item $\pi$ is an isomorphism over the nc locus of $(S,H)$;
        \item $(T,\pi^*H)$ is a semi-regular pair;
        \item no component of the non-normal locus $D_T$ of $T$ is $\pi$-exceptional.
    \end{enumerate}
We say $\pi$ is \emph{good} if additionally
\begin{enumerate}
    \item[(e)] $\Ex(\pi) \cup \pi^*H \cup D_T$ has regular components and transverse intersections.
\end{enumerate}
If $(S,H)$ is a pair, we say that a semi-regular resolution $\pi$ is \emph{thrifty} if $a(E, S, F)>-1$ for all $\pi$-exceptional divisors $E$.
\end{definition}

Note that the assumption $(S,H)$ is demi-normal implies that $\pi$ is an isomorphism over a big open set of $S$.

To show the existence of semi-regular resolutions of excellent surfaces we follow the strategy of \cite[Section 3.6]{Pos21b}.
We start with a description of involutions for complete DVR in characteristic $\neq 2$.

\begin{lemma} \label{lem: uniformise_action}
    Let $(R, \mathfrak{m})$ be a complete DVR with residue field $k:=R/\mathfrak{m}$ of characteristic $p \neq 2$.
    Let $\tau$ be a non-trivial involution of $R$ such that $\tau(\mathfrak{m})=\mathfrak{m}$.
    Then there exists a uniformizer $\pi \in \mathfrak{m} \setminus \mathfrak{m}^2$ such that $\tau(\pi)=-\pi$.
\end{lemma}

\begin{proof}
The proof is similar to \cite[Lemma 3.6.5]{Pos21b}.
We fix $t \in \mathfrak{m}$ to be a uniformiser.
     
\ptofpf{Suppose $t -\tau(t) \notin \mathfrak{m}^2.$} 
In this case, we set $\pi:=t - \tau(t)$. Note that $\tau(\pi)=\tau(t)-\tau^2(t)=\tau(t)-t=-\pi$ and we conclude.

    \ptofpf{Suppose $t -\tau(t) \in \mathfrak{m}^2.$} 
    Then there exists $f \in \mathfrak{m}$ such that $\tau(t)=(1+f)t$. Moreover, it is easy to see that $\tau(t^k)-t^k \in \mathfrak{m}^{k+1}$.
    We distinguish 2 cases.
    
    \ptofpf{1. $\tau$ acts non-trivially on $k.$}
    Let $\alpha \in k \setminus k^{\tau}$. 
    Let $\widetilde{\alpha}$ be a lifting of $\alpha$ to $R$.
    Note that $\widetilde{\alpha}-\tau(\widetilde{\alpha})(1+f)$ is invertible. Indeed, as $R$ is local, it is sufficient to note that $\widetilde{\alpha}-\tau(\widetilde{\alpha})(1+f) \equiv \alpha-\tau(\alpha) \neq 0 \mod \mathfrak{m}$ by choice of $\alpha$.
    Consider $s:= \widetilde{\alpha} t$. 
    Note that 
    $$s-\tau(s) \equiv \big(\widetilde{\alpha} - \tau(\widetilde{\alpha})\big) t \neq 0 \mod \mathfrak{m}^2 $$
    as $\alpha-\tau(\alpha) \neq 0$.
    Thus $s-\tau(s) \notin \mathfrak{m}^2$ and thus we can conclude by the previous step.

     \ptofpf{2. $\tau$ acts trivially on $k$.}
     We verify this contradicts the non-triviality of $\tau$.
     We construct a recursive sequence $t_k$ such that $t_k -\tau(t_k) \in \mathfrak{m}^{k+1}$.
     Fix $t_0=0$ and $t_1:=t$. 
     Suppose $t_k$ is defined. We have 
     $t_k -\tau(t_k)= a t^{k+1} \mod \mathfrak{m}^{k+2} $ for some $a \in R$.
     As $2$ is invertible in $R$, we can define $$t_{k+1}:=t_k - \frac{a}{2} t^{k+1}.$$ 
     Note that $\tau(\frac{a}{2}) \cong \frac{a}{2} \mod \mathfrak{m}$ by hypothesis. Therefore
     \[ \tau(t_{k+1})-t_{k+1} \equiv -a t^{k+1}+\frac{a}{2} t^{k+1} - \tau \left(\frac{a}{2} \right) t^{k+1} \mod \mathfrak{m}^{k+2} \equiv 0, \]
    as $\tau(a/2) \equiv a/2 \mod \mathfrak{m}$.
     Now the sequence $(t_k)$ is a Cauchy sequence and thus converges to $s \in \mathfrak{m}$. 
     Note that by construction $\tau(s)=s$, reaching a contradiction (as $\tau$ is non-trivial).
\end{proof}

\begin{lemma}\label{lem: gluing_nice}
    Let $(S, H)$ be a quasi-projective snc surface pair  over $T$ and let $D$ be a regular divisor intersecting transversally $H$. 
    Let $\tau \colon D \rightarrow D$ be a non-trivial involution.
    Then $(S/R(\tau), H/R(\tau))$ is a semi-regular pair. 
\end{lemma}

\begin{proof}
The relation $R(\tau)$ is finite, thus the quotient $p \colon S \to U:=S/R(\tau)$ exists and it is demi-normal by \cite[Lemma 2.3.13]{Pos21b}.
Moreover, by \cite[9.13]{kk-singbook} the diagram

	\[
	\xymatrix{
		D \ar[d] \ar[r] & S \ar[d]\\
		D/R(\tau) \ar[r] & U,
	}
	\]

is a push-out square.
We may assume $U=\Spec(R)$ is the spectrum of a local ring with maximal ideal $\mathfrak{m}_R$. 
Therefore $S$ is an affine regular scheme $\Spec(A)$, and there exists a Cartier divisor $f \in A$ (resp. $h \in A$) such that $D=(f=0)$ (resp. $H=(h=0)$). 

 If $D \to D/R(\tau)$ is a $(\mathbb{Z}/2\mathbb{Z})$-quotient we have two cases:

\begin{enumerate}
\item $A$ has exactly two maximal ideals. In this case, 
 up to an \'etale base change, we may assume $A=A_1 \oplus A_2$, where $A_1$ and $A_2$ are local rings. Let $f_i \in A_i$ (resp. $h_i \in A_i$) be the local equations of $D|_{\Spec(A_i)}$ (resp. $H|_{\Spec(A_i)}$) for $i=1,2$. 
 Note that the transversality hypothesis on $H$ and $D$ implies that $(f_i, h_i)= \mathfrak{m}_{A_i}$.
 Then the push-out property implies that $\mathfrak{m}_U=R \cap (\mathfrak{m}_{A_1} \oplus \mathfrak{m}_{A_2})$.
 Let $\tau \colon A_1/f_1 \to A_2/f_2$ be the involution. 
therefore $x:=(f_1,0), y=(0,f_2)$ and $z:=(h_1,h_2)$ generate $\mathfrak{m}_U$. 
If $\tau$ is trivial, then we have the relation $x=y$ and $H/R(\tau)=(z=0)$ and we are in case \autoref{item: regular} of \autoref{def: semi-regular-pair}. 
If $\tau$ is not-trivial, we have the relation $xy=0$ and therefore
 $\Spec(R)$ is nodal and $H/R(\tau)=(z=0)$, thus ending in case \autoref{item: nodal} of  \autoref{def: semi-regular-pair}. 

\item $A$ is a local ring such that $\tau(\mathfrak{m}_A)=\mathfrak{m}_A$ and let $\tau \colon A/(f) \to A/(f)$ be an involution.
The involution $\tau$ extends to the completion of $A$ and the completion of $R$ is the preimage of the $\tau$-invariant elements of $A/(f)$. 
As $\tau$ fixes $\mathfrak{m}_A$, 
the residue field of $R$ is isomorphic to $k^{\tau}$.
The completion $\hat{A}/(f)$ is a complete DVR and thus there exists a uniformiser $\pi \in \hat{A}/(f)$ such that $\tau(\pi)=-\pi$ by \autoref{lem: uniformise_action}.
Let $\widetilde{\pi}$ be a lifting of $\pi$ to $\hat{A}$ such that $H=(\widetilde{\pi}=0)$ (note that $h/\widetilde{\pi}$ is invertible).

We distinguish two cases.
\begin{enumerate}
    \item Suppose $k=k^{\tau}$.
    Then $R \subset A$ is the subalgebra generated by $f, \widetilde{\pi}$ and $\widetilde{\pi}f$. 
    Moreover, $H/R(\tau)$ is  given by the equations $ (\widetilde{\pi}=\widetilde{\pi}f=0) $ and thus we are in case \autoref{item: pinch} of \autoref{def: semi-regular-pair}.
    \item Suppose $k^{\tau} \subsetneq k$, there exists $\alpha \in k$ such that $k=k^{\tau}(\alpha)$ and $\tau(\alpha)=-\alpha.$ Let $A'$ to be the preimage of $k^{\tau}$ under the projection $A \to k$.
    In this case consider the subalgebra of $A'$ generated by $x:=\alpha \widetilde{\pi}, y:=f, z:= \alpha f$. Therefore we have the relation $ \alpha^2y^2=z^2$, showing $R$ is a nodal singularities and $H/R(\tau)$ is described by $(x=0)$, showing we end up in case \autoref{item: nodal} of \autoref{def: semi-regular-pair}.
\end{enumerate}
\end{enumerate}
\end{proof}

We now show the existence of semi-resolution (in characteristic 0, this is \cite[Theorem 10.54]{kk-singbook}).

\begin{corollary}\label{prop: exist-semi-log-res}
Let $(S, H)$ be a quasi-projective demi-normal surface pair over $T$.
Then there exists a good semi-regular resolution $\pi \colon V \to (S,H)$.
If moreover $(S, H)$ is semi-dlt, we can choose $\pi$ to be thrifty.
\end{corollary}
	
	\begin{proof}
        Let $\nu \colon (S^\nu,  D_{S^\nu}) \to S$ be the normalisation morphism where $D_{S^\nu}$ is the conductor subscheme of $S^\nu$.
        Let $f\colon X \to (S^\nu,  D_{S^\nu}+\nu^*H)$ be a log resolution of $(S^\nu,  D_{S^\nu}+\nu^*H)$ and let $D_X:=f_*^{-1}D_{S^\nu}$.
        The involution $\tau$ lifts to an involution of $D_X$  
        and we can apply \autoref{lem: gluing_nice} to construct a projective birational contraction $q \colon Y \to S$ fitting in the commutative diagram
        	\[
	\xymatrix{
		(X, D_X+f_*^{-1}H+\Ex(f)) \ar[d]_{f} \ar[r] & (Y:=X/R(\tau),  (f_*^{-1}H+\Ex(f))/\tau) \ar[d]_{q}\\
		(S^{\nu}, D_{S^\nu}+\nu^*H)  \ar[r] & (S,H),
	}
	\]
    such that $q \colon Y \to (S,H)$ is a semi-regular resolution of $(S,H)$.
    
    If $(S,H)$ is semi-dlt, then $(S^\nu, D_{S^n}+\nu^*H)$ is a dlt pair. 
    In this case, we can take $f \colon X \to (S^\nu, D_{S^n}+\nu^*H)$ to be a thrifty log resolution of $(S^n, D_{S^n})$ and the end-product $T \to (S,H)$ is clearly a thrifty semi-regular resolution.
	\end{proof} 

We show how we can slightly improve the resolution algorithm (see \cite[Corollary 10.55]{kk-singbook} for an analogue in characteristic 0).

\begin{definition}
	Let $(S, H)$ be a demi-normal surface pair and let $S^0 \subset S$ be the largest open set such that $(S^0, H|_{S^0})$ is semi-snc. 
    We say that $\pi \colon T \to (S,F)$ is a \emph{semi-log resolution} if 
	\begin{enumerate}
    \item $\pi$ is projective and birational; 
    \item $(T, D_T:=\pi_*^{-1}\Supp(F)+\Ex(\pi))$ is a semi-snc pair;
    \item $\pi$ is an isomorphism over the generic point of every lc centre of $(S,H)$;
    \item $\pi$ is an isomorphism at the generic point of every lc centre of $(T, D_T)$.
    \end{enumerate}
\end{definition}

\begin{theorem}\label{thm: exist-semi-log-res}
    Let $(S, H)$ be a quasi-projective demi-normal surface pair over $T$.
    Then there exists a semi-log resolution $\pi \colon V \to (S,H)$.
    If $(S, H)$ is semi-dlt, we can choose $\pi$ to be thrifty.
\end{theorem}

\begin{proof}
Consider $q \colon Y \to (S,H)$ be the semi-regular resolution constructed in \autoref{prop: exist-semi-log-res}. 
The only problem is around the pinch points of $Y$, which are isolated by dimension reasons. 
Therefore we can localise to a neighbourhood of $y \in Y$, where $y$ is a pinch point and let $D_y$ be the local component of the non-normal locus $D_T$.
By blowing-up $D_y$ we obtain our desired semi-log resolution as explained in \autoref{rem: blow-up-non-normal-locus-pinch}.
\end{proof}
	
\subsection{Vanishing theorems for slc surfaces}

In this section we generalise the vanishing theorems of Kawamata--Viehweg type for klt surfaces due to Tanaka \cite[Theorem 3.3]{Tan18} to the slc case using the method developed by Koll\'{a}r in \cite[Section 10.3]{kk-singbook}. 
For an overview on vanishing theorems for slc pairs in characteristic 0 we refer to \cite{Fuj15}.
The most general result we prove is \autoref{c-vanishing-slc-surf2}, which is the fundamental vanishing theorem we will use in \autoref{s-depth-3fold}.
We start with the case of semi-snc surface pairs.

\begin{proposition} \label{p-kvv-vanishing-snc}
	Let $(S, \Delta)$ be a semi-snc surface pair  with $\Delta$  a reduced $\bZ$-divisor, and let $f \colon S \to T$ be a surjective projective morphism onto a normal scheme of dimension $\dim(T) \geq 1$.
	Let $M$ be a Cartier divisor on $S$. Suppose that
	\begin{enumerate}
		\item $M \sim_{\Q, f} K_S+\Delta' + L$, where
		\begin{enumerate}
		    \item 
			 $L$ is a $f$-nef $\mathbb{Q}$-divisor, and
			 \item  $0 \leq \Delta' \leq \Delta$ is a $\bQ$-divisor;
			 \end{enumerate} 
		\item if $Z$ is a log canonical  centre of  $\left(S, \Delta \right)$, including the irreducible components of $Z$ as well, then   
		\begin{enumerate}
			\item $\dim \big( f(Z) \big) \geq 1$;
			\item if $F_Z$ is the generic fibre of $Z \to f(Z)$, then $\dim F_Z =\nu(L|_{F_Z})$.
			\end{enumerate}
	\end{enumerate}
	Then $R^1f_*\MO_S(M)=0$.
\end{proposition}
	
\begin{proof}
	Using the relative Kawamata--Viehweg vanishing theorem for surfaces \cite[Theorem 3.3]{Tan18}, we can repeat the same steps of the proof of \cite[Corollary 10.34]{kk-singbook}.
\end{proof}
	
	The following result is useful to reduce various statements to the case of semi-snc pairs.

\begin{lemma}\label{l-easy-descend-vanishing}
		Let $(S, \Delta)$ be a semi-dlt surface pair such that $\frac{1}{2} \in \mathcal{O}_S$. 
		Let $D$ be a $\Q$-Cartier $\mathbb{Z}$-divisor such that $D \sim_{\Q} K_S+\Delta+M$ for some $\Q$-Cartier divisor $M$.
		Then there is a proper birational morphism $g \colon Y \to S$, a Cartier divisor $D_Y$ and a $\Q$-divisor $\Delta_{D, Y}$ such that
		\begin{enumerate}
			\item $(Y, \Delta_{D, Y})$ is semi-snc;
			\item $D_Y \sim_{\mathbb{Q}} K_Y+\Delta_{D,Y}+g^*M;$ 
			\item  if $Z$ is a log canonical  centre of $(Y, \Delta_{D, Y})$, then the restriction $g|_Z$ is birational.	
			\item $g_*\MO_Y(D_Y)=\MO_S(D)$;
			\item $R^ig_*\MO_Y(D_Y)=0$ for $i>0$.
		\end{enumerate}
	\end{lemma}
	
	\begin{proof}
		The same proof of \cite[Proposition 10.36]{kk-singbook} applies as thrifty semi-log resolutions by \autoref{thm: exist-semi-log-res} exist for excellent surfaces and the necessary vanishing theorems hold by \autoref{p-kvv-vanishing-snc}.
	\end{proof} 
	
	We generalise \autoref{p-kvv-vanishing-snc} to the case of semi-dlt surface pairs.
	
	\begin{proposition} \label{p-kvv-vanishing-sdlt-surf}
		Let $(S, \Delta)$ be a semi-dlt surface pair such that $\frac{1}{2} \in \mathcal{O}_S$ and let $f \colon S \to T$ be a projective morphism onto a normal scheme of $\dim(T) \geq 1$.
		Let $D$ be a $\Q$-Cartier $\Z$-divisor on $S$. Suppose that
		\begin{enumerate}
			\item $D \sim_{\Q, f} K_S+\Delta+ L$, where $L$ is a $f$-nef $\mathbb{Q}$-divisor;
		    \item if $Z$ is a log canonical centre of of $(S, \Delta)$,  including the irreducible components of $Z$ as well, then 
		    \begin{enumerate}
		    \item
		    $\dim(f(Z)) \geq 1$;
			\item if $F_Z$ is the generic fibre of $Z \to f(Z)$, then $\dim F_Z =\nu(L|_{F_Z})$.
			\end{enumerate}
		\end{enumerate}
		Then $R^1f_*\MO_S(D)=0$.
	\end{proposition}
	
	\begin{proof}
		Let  $g \colon Y \to S$ be a proper birational morphism such that the pair $(Y, \Delta_{D,Y})$ and the Cartier divisor $D_Y$ on $Y$ satisfy the conditions of \autoref{l-easy-descend-vanishing} and denote by $h \colon Y \to T$ the natural composition.
		By \autoref{p-kvv-vanishing-snc} we conclude that $R^1h_*\mathcal{O}_Y(D_Y)=0$. As $g_*\MO_Y(D_Y)=\MO_S(D)$ and $R^ig_*\MO_Y(D_Y)=0$ for $i>0$ by \autoref{p-kvv-vanishing-snc}. We deduce $R^1f_*\MO_S(D)=0$ by the Leray spectral sequence.
	\end{proof}
	
	In order to generalise to the slc case, the following is a useful observation.

	\begin{lemma}\label{l-slc-to-sdlt}
	    Let $(S, \Delta)$ be an slc surface pair  such that $\frac{1}{2} \in \mathcal{O}_S$  and let $f \colon S \to T$ be a projective morphism onto a normal scheme of $\dim(T) \geq 1$ .
	    Suppose that
	    \begin{enumerate}
	        \item every irreducible component of $S$ is $(R_1)$;
	        \item if $Z$ is a log canonical  centre of $(S, \Delta)$, then $\dim(f(Z)) \geq 1$.
	    \end{enumerate}
	Then $(S, \Delta)$ is semi-dlt.
	\end{lemma}
	
	\begin{proof}
	    We argue by contradiction. Let $E$ be an exceptional divisor such that $a(E, S, \Delta)=-1$ and $\cent_S(E) \subset \nsnc(S, \Delta)$.
	    The hypothesis (a) guarantees that $(S,\Delta)$ is snc at codimension 1 points. Therefore $\cent_S(E)$ is a closed point, contradicting (b).
	\end{proof}

	\begin{corollary}
	\label{c-kvv-vanishing-slc-surf}
	    Let $(S, \Delta)$ be an slc surface pair such that $\frac{1}{2} \in \mathcal{O}_S$ and let $f \colon S \to T$ be a projective morphism onto a normal scheme of $\dim(T) \geq 1$.
		Let $D$ be a $\Q$-Cartier $\Z$-divisor on $S$. Suppose that
		\begin{enumerate}
			\item $D \sim_{\Q, f} K_S+\Delta+ L$, where $L$ is $f$-nef;
            \item if $Z$ is a log canonical  centre of $(S, \Delta)$,   including the irreducible components of $Z$ as well, then 
            \begin{enumerate}
            \item $\dim(f(Z)) \geq 1$;
			\item if $F_Z$ is the generic fibre of $Z \to f(Z)$, then $\dim F_Z =\nu(L|_{F_Z})$.		\end{enumerate}
            \end{enumerate}
		Then $R^1f_*\MO_S(D)=0$.
	\end{corollary}
	
	\begin{proof}
	Let $p \colon \widetilde{S} \to S$ be the double cover of \autoref{l-2cover-slc}.
	As $2$ is invertible, $\mathcal{O}_S(D)$ is a direct summand of $p_*\mathcal{O}_{\widetilde{S}}(p^*D)$. We can thus assume that the irreducible components of $S$ are regular in codimension 1.
	In this case $(S, \Delta)$ is semi-dlt by \autoref{l-slc-to-sdlt} and we conclude by \autoref{p-kvv-vanishing-sdlt-surf}.
	\end{proof}
	
We can prove a further generalisation of Kawamata--Viehweg vanishing for slc surfaces over curves.

	\begin{proposition}\label{p-vanishing-ssnc-surf}
		Let $(S, \Delta)$ be a semi-snc surface pair with $\Delta$ a reduced divisor, and let $f \colon S \to C$ be a projective  surjective contraction onto a normal curve $C$.
		Let $D$ be a $\Q$-Cartier $\bZ$-divisor on $S$. 
		Suppose that
		\begin{enumerate}
			\item every log canonical  centre $Z$ of $(S, \Delta)$, including the components of $S$, dominates $C$.
			\item  $A$ is an $f$-nef $\Q$-Cartier $\mathbb{Q}$-divisor on $S$;
			\item  $D \sim_{f, \Q} K_S+\Delta'+A$, where $0 \leq \Delta' \leq \Delta$ is a $\bQ$-divisor;
			\item on every connected component $S'$ of $S$ there exists an irreducible component $E$ of $S'$ such that $A|_{E}$ is $f|_E$-big.
		\end{enumerate}
		Then $R^1f_*\MO_{S}(D)=0$.
	\end{proposition}
	
\begin{proof}
    We prove the result by induction on the number $n$ of irreducible components of $S$.
	If $n=1$, we conclude by \autoref{p-kvv-vanishing-snc}.
	
    We prove the induction step. Let $E$ be an irreducible component of $S$ such that $A|_E$ is big over $C$. 
    Let $T$ be the union of the irreducible components of $S$ except $E$. 	
	Denote $B:=E \cap T$ and we consider the short exact sequence:
	$$0 \to \mathcal{O}_{E}(D|_{E}-B) \to \mathcal{O}_S(D) \to \mathcal{O}_{T}(D|_{T}) \to 0. $$
	Taking the long exact sequence in cohomology, it is sufficient to show that $R^1g_*\MO_{E}(D-B)=R^1g_*\MO_{T}(D|_{T})=0$ to conclude that $R^1g_*\MO_S(D)=0$.
		
	Since $K_S|_{E}=K_{E}+B$ we have 
	$  D|_{E}-B \sim_{g, \Q} K_{E} +\Delta'|_{E}+A|_{E};$
	so by \autoref{p-kvv-vanishing-snc} we conclude that $R^1g_*\MO_{E}(D-B)=0$.
	
	Since $K_S|_{T}=K_{T}+B$ we have
	$ D_{T} \sim_{\Q, g} K_{T}+\Delta'|_{T}+B+A|_{T},$
	and $B$ is not trivial on some irreducible component of every connected component of $T$, as $g$ has connected fibers. By hypothesis, $B$ must be a non-empty horizontal divisor, thus it is nef over $C$ and for every connected component of $T$ there exists an irreducible component $F$ such that $B|_F$ is $g|_F$-big. Therefore we apply the induction hypothesis to deduce $R^1g_*\MO_{T}(D|_{T})=0$.		
\end{proof}

	\begin{proposition}\label{c-vanishing-sdlt-surf}
		Let $(S, \Delta)$ be a semi-dlt  surface pair such that  $\frac{1}{2} \in \mathcal{O}_S$ , let $g \colon S \to C$ be a projective morphism onto a normal curve $C$.
		Let $D$ be a $\Q$-Cartier $\bZ$-divisor on $S$. 
		Suppose that
		\begin{enumerate}
			\item every log canonical  centre $Z$ of $(S, \Delta)$, including the components of $S$, dominates $C$;
			\item  $A$ is a $\Q$-Cartier $\Q$-divisor on $S$, which is $g$-nef;
			\item  $D \sim_{g, \Q} K_S+\Delta'+A$, where $0 \leq \Delta' \leq \Delta$;
			\item on every connected component of $S$, there exists an irreducible component $E$ such that $A|_{E}$ is $g|_E$-big.
		\end{enumerate}
		Then $R^1g_*\MO_{S}(D)=0$.
	\end{proposition}
	
	\begin{proof}
    As in the proof of \autoref{p-kvv-vanishing-sdlt-surf}, it sufficient to combine \autoref{l-easy-descend-vanishing} and \autoref{p-vanishing-ssnc-surf} with the Leray spectral sequence.
	\end{proof}

\begin{theorem}\label{c-vanishing-slc-surf2}
		Let $(S, \Delta)$ be a slc surface pair such that $\frac{1}{2} \in \mathcal{O}_S$, let $g \colon S \to C$ be a projective morphism onto a normal curve $C$.
		Let $D$ be a $\Q$-Cartier $\bZ$-divisor on $S$. 
		Suppose that
		\begin{enumerate}
			\item every log canonical  centre  $Z$ of $(S, \Delta)$, including the components of $Z$, dominates $C$;
			\item  $A$ is a $\Q$-Cartier $\Q$-divisor on $S$, which is $g$-nef;
			\item  $D \sim_{g, \Q} K_S+\Delta'+A$, where $0 \leq \Delta' \leq \Delta$.;
			\item on every connected component of $S$, there exists an irreducible component $E$ such that $A|_{E}$ is $g|_E$-big;
		\end{enumerate}
		Then $R^1g_*\MO_{S}(D)=0$.
	\end{theorem}
	
\begin{proof}
We can repeat the same proof of \autoref{c-kvv-vanishing-slc-surf} using \autoref{c-vanishing-sdlt-surf}.
\end{proof}
	
\subsection{Grauert--Riemenschneider theorem for dlt 3-folds}

	We recall the Grauert--Riemenschneider (GR) vanishing theorem for excellent dlt 3-folds proven by Koll\'ar and the second author in \cite{BK20}.
	
	\begin{theorem}[{\cite[Theorem 2]{BK20}}]\label{t-gr}
		Let $(X, \Delta)$ be a 3-dimensional dlt pair whose residue fields of closed points are perfect with characteristic $p \neq 2,3, 5$. Then GR vanishing holds on $(X, \Delta)$. 
		
		Precisely, let  $g \colon Y \to (X, \Delta)$ be a log resolution,  and let  $D$ be a Weil $\mathbb{Z}$-divisor on $Y$  such that $D \sim_{g,\mathbb{R}} K_{Y}+\Delta'$ for an  effective $\mathbb{R}$-divisor $\Delta'$  on $Y$ such that  $g_*\Delta'\leq \Delta$ and
		$\lfloor \Ex(\Delta') \rfloor=0$.
		Then,  $R^ig_*\MO_{Y}(D)=0$ for $i>0$.
	\end{theorem}
	
	The main techniques are the vanishing theorem for surfaces of del Pezzo type over perfect fields proven in \cite{ABL20} and the MMP for 3-folds \cite{7authors}.
	From the G-R vanishing theorem one can deduce various rationality and Cohen--Macaulay properties for dlt 3-fold singularities, a result we will frequently use to study depths of log canonical  3-fold singularities in terms of a dlt modifications.

	\begin{corollary}[{\cite[Theorem 17]{BK20}}]\label{c-rat-dlt}
		Let $(X, \Delta)$ be a 3-dimensional dlt pair whose residue fields of closed points are perfect with characteristic $p \neq 2,3, 5$.
		Then
		\begin{enumerate}
			\item $X$ is Cohen--Macaulay, and has rational singularities;
			\item every irreducible component of $\lfloor \Delta \rfloor$ is normal;
			\item \label{itm:c-rat-dlt:divisor_CM} if $D$ is a $\mathbb{Z}$-divisor such that $D+\Delta'$ is $\Q$-Cartier for some $\Q$-divisor $0 \leq \Delta'\leq \Delta$, then $\mathcal{O}_X(D)$ is CM.
		\end{enumerate}
	\end{corollary} 

\section{Depth of log canonical 3-fold singularities}\label{s-depth-3fold}

\begin{setting}	
\label{setting:runing_MMP}
	Throughout this section, we suppose $(R, \mathfrak{m})$ is a local ring whose residue field is \emph{perfect} of characteristic $p \neq 2,3, \text{ and } 5$.
	Let $(X=\Spec(R), x)$ be the associated local scheme and suppose that there exists a $\mathbb{Q}$-divisor $\Delta \geq 0$ such that $(X,\Delta)$ is a log canonical 3-dimensional pair.
	\end{setting}
	
	We are interested in computing the local cohomology group $H^2_x(X,\mathcal{O}_{X})$ when the closed point $x$ is \emph{not} a minimal log canonical  centre of $(X,\Delta)$.
    We first that show we can reduce to the case where the minimal log canonical  centre is 1-dimensional.

	\begin{lemma}\label{l-not1d-minimal-lc}
		Let $C$ be the minimal log canonical  centre of $(X,\Delta)$ passing through $x$.
		Suppose one of the following conditions hold:
		\begin{enumerate}
			\item $x \notin \nklt(X,\Delta)$ (that is, $C$ is empty),
			\item $x \in \nklt (X,\Delta)$ and $\dim(C)=2$. 
		\end{enumerate}
		Then $X$ satisfies Serre's condition $(S_3)$.
	\end{lemma}
	\begin{proof}
		Case (a) is proven in \autoref{c-rat-dlt}. 
		In case (b), we deduce $(X, \Delta)$ is plt by \autoref{c-minimal-cod1} and we conclude by \autoref{c-rat-dlt}.
	\end{proof}

By \autoref{l-not1d-minimal-lc} the case of interest, when studying the behavior of local cohomology of $\sO_X$, is when $\dim C =1$. In this case, our main technical result relates the non-vanishing of local cohomology $H^2_x(X, \mathcal{O}_X)$ to the torsion of $R^1g_*\mathcal{O}_E$, where $E$ is an exceptional divisor over $X$. More precisely:

\begin{theorem}\label{t-hyp-1-good-model}
	Let $C \subset X$ be a 1-dimensional minimal log canonical  centre for $(X,\Delta)$ passing through $x$. 
	Then there exists a projective birational morphism $g \colon Z \to (X, \Delta)$ with reduced exceptional divisor $E$ such that
	\begin{enumerate}
	    \item $Z$ is $\mathbb{Q}$-factorial klt with $K_Z+g_*^{-1}\Delta+E \sim_{\bQ} g^*(K_X+\Delta)$;
		\item $E$ is $(S_2)$;
		\item \label{itm: id_h2_h0} $H^2_x(X, \MO_X) \simeq H^0_x(C, R^1g_*\mathcal{O}_E)$.
	\end{enumerate}
\end{theorem}

As the rest of the section is devoted to showing \autoref{t-hyp-1-good-model}, from now on we assume the following:
	
	\begin{setting}\label{s-lc-curve}
		Besides the assumptions and the notation of \autoref{setting:runing_MMP}, let us also fix  
		a 1-dimensional minimal log canonical  centre $C$ of $(X,\Delta)$ passing through $x$.
		Moreover, $C$ is irreducible by \autoref{l-inters-lcc} and we denote by $\eta$ its generic point. 
	\end{setting}

\subsection{Construction of minimal \'etale-dlt modifications}

The hypothesis of minimality on $C$ allows to prove the following technical results, which we will use repeatedly.

\begin{lemma}\label{l-trick-dlt}
	Let $\pi \colon (Y, \Delta_Y) \to (X, \Delta)$ be a crepant proper birational morphism of normal log pairs where $(X,\Delta)$ is as in \autoref{s-lc-curve}. 
	Suppose $Y$ is $\mathbb{Q}$-factorial and let $0 \leq  \Gamma \leq \Delta_Y$.
	If $(Y_{\eta}, \Gamma_{\eta})$ is dlt (resp. \'etale-dlt), then $(Y, \Gamma)$ is dlt (resp. \'etale-dlt).
\end{lemma}
	
\begin{proof}
	Suppose $(Y_\eta, \Gamma_\eta)$ is dlt.
	If $(Y,\Gamma)$ is not dlt, then there exists an exceptional divisor $E$ with discrepancy $a(E, Y, \Gamma)=-1$ such that $\cent_Y(E) \subset Y \setminus \snc(Y,\Gamma) $. 
	Since $(Y_\eta, \Gamma_\eta)$ is dlt, we deduce that $\cent_Y(E)$ must be disjoint from $Y_\eta$. In particular, $\cent_X(E)$ is a closed point $c$ in $C$.
	As $a(E, Y, \Delta_Y) \leq a(E, Y, \Gamma)=-1$, this contradicts the minimality of $C$ among the log canonical  centres of $(X, \Delta)$. 
	 
    The same proof works in the \'etale-dlt case by replacing $Y \setminus \snc(Y,\Gamma)$ with the closed subset $Y \setminus \etsnc(Y,\Gamma)$.
\end{proof}

\begin{lemma}\label{lem: centre-lc-places}
    Let $(X,\Delta)$ as in \autoref{s-lc-curve}. 
    For every exceptional log canonical place $E$ over $X$, we have $\cent_X(E)=C$.
\end{lemma}

\begin{proof}
    If $\cent_X(E) \neq C$, we have $\cent_X(E) \cap C= \left\{ x\right\}$ as $E$ is exceptional. Therefore $x$ is a log canonical centre by \autoref{l-inters-lcc}, contradicting the minimality of $C$.
\end{proof}

In the next propositions, as in the article in general,  $\Ex(\pi)$ denotes the divisorial part of the exceptional set of a proper birational morphism $\pi$, not the entire exceptional set.

\begin{lemma}\label{lem: epsilon-plt}
Let $\pi \colon (Y, \Delta_Y) \to (X, \Delta)$ be a crepant proper birational morphism of normal log pairs where $(X,\Delta)$ is as in \autoref{s-lc-curve}. 
Suppose $Y$ is $\mathbb{Q}$-factorial and $\Delta_Y \geq E:=\Ex(\pi)$. 
If $E \neq 0$, then the pair $(Y, \Delta_Y-\varepsilon E)$ is plt for every rational number $\varepsilon \in (0,1]$.
\end{lemma}

\begin{proof}
Note that, as $\pi_\eta \colon Y_\eta \to \Spec(\mathcal{O}_{X,\eta})$ is a proper birational morphism of normal surfaces, the support of $E_\eta$ coincides with $\Ex(\pi)_\eta$ also set-theoretically.

Write $\Delta_Y-\varepsilon E=\Delta_Y'+(1-\varepsilon)E \leq \Delta_Y$. where $\Delta_Y'$ and $E$ are effective $\mathbb{Q}$-divisor and they have no irreducible components in common.
Suppose by contradiction that $\big(Y, \Delta_Y'+(1-\varepsilon)E \big)$ is not plt. 
By definition, there exists a proper birational modification $f \colon Z \to Y$ extracting an exceptional divisor $F$ with discrepancy $a(F, Y, \Delta_Y'+(1-\varepsilon)E)=-1$.
By the monotonicity of discrepancies \cite[Lemma 2.27]{km-book},  
$$a(F, Y, \Delta_Y) \leq a(F, Y, \Delta_Y'+(1-\varepsilon)E)=-1.$$ 
As $(Y, \Delta_Y)$ is log canonical, we conclude $a(F,Y, \Delta_Y)= -1$.
As $F$ is an exceptional log canonical place over $X$, then $\cent_X(F)=C$ by \autoref{lem: centre-lc-places}.
As $\cent_Y(F)$ dominates $C$, we deduce that $\cent_Y(F) \subset E$ (as $E_\eta$ coincides set-theoretically with $\Ex(\pi)_\eta$).
This last containment implies by \cite[Lemma 2.5]{kk-singbook} that 
$$a(F,Y, \Delta_Y'+(1-\varepsilon)E)=a(F, Y,\Delta_Y -\varepsilon E) > a(F, Y, \Delta_Y)=-1,$$
contradicting the starting assumption $a(F, Y, \Delta_Y-\varepsilon E)=-1$.
\end{proof}

\begin{proposition}\label{dltlc-minimal-modification} Let $(X,\Delta)$ be as in \autoref{s-lc-curve}.
    Then there exists a projective birational morphism $g \colon Z \to X$ such that
    \begin{enumerate}
        \item \label{itm:dltlc-minimal-modification:lc} $\big(Z, g_*^{-1}\Delta+\Ex(g)\big)$ is a $\mathbb{Q}$-factorial log canonical  pair such that $K_Z + g_*^{-1}\Delta+\Ex(g) = g^*(K_X+\Delta)$;
        \item \label{itm:dltlc-minimal-modification:etale_dlt} the pair $(Z, g_*^{-1}\lfloor \Delta \rfloor+\Ex(g))$ is \'etale-dlt;
        \item \label{itm:dltlc-minimal-modification:klt} for every $\varepsilon>0$, the pair $\big(Z, g_*^{-1}\Delta+(1-\varepsilon)\Ex(g) \big)$ is plt;
        \item \label{itm:dltlc-minimal-modification:nef} $-\Ex(g)$ is a $g$-nef  $\mathbb{Q}$-Cartier divisor;
        \item \label{itm:dltlc-minimal-modification:surjects_onto_C} $g(F)=C$ for every irreducible component $F$ of $\Ex(g)$.
    \end{enumerate}
\end{proposition}

\begin{proof} 
    Let $\varphi \colon  W \to X$ be a log resolution of $(X, \Delta)$ such that $\varphi_*^{-1}\Delta$ is regular.
    In particular, the pair $(W, \varphi_*^{-1}\Delta )$ is plt. 
    Let $\pi \colon Y \to X$ be a log minimal model of this pair over $X$, which is $\bQ$-factorial by the plt assumption. 
    Write 
 \begin{equation*}
 K_Y + \pi^{-1}_* \Delta + E + B =\pi^*(K_X+\Delta),
 \end{equation*} 
    where $E$ is an effective $\mathbb{Z}$-divisor and $\lfloor B \rfloor=0$. In particular, then $-(E+B)$ is a $\mathbb{Q}$-Cartier nef divisor over $X$. 
    We denote by $Y_\eta$ the base change of $Y$ over $X_{\eta}:=\Spec(\mathcal{O}_{X, \eta})$ and for a divisor $D$ on $Y$, we will denote by $D_\eta$ the localisation $D|_{Y_\eta}$. In particular, $Y_\eta$ is NOT the fiber over $\eta$. 

    By \autoref{rem-construction-log-mininal-resolution}, $Y_\eta$ is a log minimal resolution of the surface $\big(\Spec \mathcal{O}_{X, \eta}, \Delta_\eta \big)$.
    Let $G$ be the $\mathbb{Z}$-divisor $\Ex(\pi)-E$, which is supported on the exceptional divisors which are not log canonical places. Note that $\Supp(B) \subseteq \Supp(G)$.
 
Next, we define a birational model $h : V \to X $ that satisfies the following properties, where the sub-index $V$ denotes the strict transform of the corresponding divisor:
\begin{itemize}
\item $V$ is $\bQ$-factorial,
\item $\Ex(h)=E_V$, 
\item $-E_{V, \eta}$ is nef,
\item $\big(V, E_V + h^{-1}_* \lfloor \Delta \rfloor \big)_{\eta}$ is \'etale-dlt. 
\end{itemize}
In particular, as $\Ex(h)=E_V$, we conclude $(V, h_*^{-1}\Delta+E_V)$ is crepant birational to $(X,\Delta)$.
We construct $V$ separately in the two cases corresponding to the two points of   \autoref{p-lmn-etale-dlt},   when \autoref{p-lmn-etale-dlt} is  applied to to the minimal resolution $\pi_\eta: Y_\eta \to \big( X_\eta, \Delta_\eta \big)$. 
 
     \ptofpf{Case \autoref{itm:p-lmn-etale-dlt:one_exceptional} of \autoref{p-lmn-etale-dlt}:} in this case we have that $\Delta_\eta=G_{\eta}=0$ and that $E_\eta$ is equal to $\Ex(\pi_\eta)$, it is irreducible, and it is anti-nef. Note that the pair $(Y, \pi^{-1}_* \Delta + E + B)$ is $\mathbb{Q}$-factorial log canonical, crepant to $(X, \Delta)$ and by \autoref{lem: epsilon-plt} the pair $ \big(Y, \pi^{-1}_* \Delta + (1- \varepsilon)E + B\big)$ is plt. 
     Additionally, as $G$ is exceptional, it does not have any of the codimension $1$ components of $\pi^{-1}_* \Delta$ in its support. This implies that we may find another rational number $\varepsilon'>0$, such that $ \big(Y, \pi^{-1}_* \Delta + (1- \varepsilon)E + B + \varepsilon' G\big)$ is still plt. Let $V$  be a log minimal model over $X$ of the latter pair. Note that we have 
     \begin{equation}
 \label{eq:p-lmn-etale-dlt:negativity_elliptic}    
K_Y+\pi_*^{-1}\Delta + (1- \varepsilon)E+B+\varepsilon' G \equiv_X \varepsilon' G - \varepsilon E 
\end{equation}
In particular,  this MMP is the identity on $Y_\eta$, as 
\begin{equation*}
\big(\varepsilon G - \varepsilon E \big)_{\eta} = -\varepsilon E_\eta
\end{equation*}
is nef. This also yields that $-E_{V, \eta}$ is nef. It even implies $\big(V, E_V + h^{-1}_* \lfloor \Delta \rfloor \big)_{\eta}$ is \'etale-dlt, as $\big(Y, E + \pi^{-1}_* \lfloor \Delta \rfloor \big)_{\eta}$ is \'etale-dlt by point \autoref{itm:p-lmn-etale-dlt:one_exceptional} of \autoref{p-lmn-etale-dlt}.  
Additionally, by the negativity lemma and by \autoref{eq:p-lmn-etale-dlt:negativity_elliptic}, this MMP turns $G$ anti-effective, which means that it contracts it. Hence, $\Ex(h)=E_V$. 
Finally, $V$ is $\bQ$-factorial as it is a result of a plt MMP. 

\ptofpf{Case \autoref{itm:p-lmn-etale-dlt:general} of \autoref{p-lmn-etale-dlt}:} by point \autoref{itm:p-lmn-etale-dlt:general} of \autoref{p-lmn-etale-dlt}, we know that $(Y, \pi^{-1}_* \lfloor{\Delta \rfloor}  + E + B)_\eta$ is dlt. Hence, by \autoref{l-trick-dlt}, $(Y, \pi^{-1}_* \lfloor{\Delta \rfloor} + E + B)$ is also dlt. As the coefficients of $B$ are smaller than $1$, we may choose  a  rational number $\varepsilon >0$ such that 
     $\big(Y, \pi_*^{-1}\lfloor{\Delta \rfloor}+ E+B+ \varepsilon G\big)$ is dlt. Let $h : V \to X$ be a log minimal model of this latter pair over $X$, where $V$ is $\bQ$-factorial as we run a dlt MMP on a $\bQ$-factorial variety.  Note that we have
\begin{equation*}
K_Y+\pi_*^{-1}\lfloor{\Delta \rfloor}+ E+B+\varepsilon G \equiv_X \varepsilon G  - \pi_*^{-1}\left\{ \Delta \right\}.
\end{equation*}
Therefore, by the negativity lemma \cite[Lemma 2.16]{7authors} this MMP turns $G$ anti-effective, which means that it contracts it, and hence it also contracts $B$. Hence $E_V=\Exc(h)$, and $(V, h^{-1}_* \lfloor \Delta \rfloor + E_V)$ is dlt. The last property we need to show about $V$ is that $-E_{V,\eta}$ is nef. By \autoref{l-pushfor-nef}, it is enough to show for this that $-E_\eta - B_\eta$ is nef. However, that is immediate as $-E-B \equiv_X K_Y + \pi^{-1}_* \Delta$, which is nef by the construction of $Y$ as a log minimal model.

Having finished  the construction and the verification of the properties of $V$ in both cases, by \autoref{lem: epsilon-plt} the pair $(V, h^{-1}_* \Delta + (1- \varepsilon) E_V)$ is plt for every small rational number $\varepsilon>0$. 
Let $g\colon  Z \to X$ be the log minimal model of this latter pair, which is $\bQ$-factorial. 
This yields point \autoref{itm:dltlc-minimal-modification:lc} and  \autoref{itm:dltlc-minimal-modification:klt}. 
As 
\begin{equation*}
K_V + h^{-1}_* \Delta + (1- \varepsilon) E_V \equiv_X - \varepsilon E_V,
\end{equation*}
we see that this MMP is the identity in a neighborhood of $V_\eta$ since $-E_{V,\eta}$ is nef, and that $-E_Z=- \Ex(g)$ is nef. This yields \autoref{itm:dltlc-minimal-modification:nef}. 
The MMP being identity over $\eta$ also implies that the exceptional divisor of $Z$ and $V$ are the same,  and that $(Z, g^{-1}_* \lfloor \Delta \rfloor + E_Z)_\eta$ is \'etale-dlt.  
The former yields point \autoref{itm:dltlc-minimal-modification:surjects_onto_C} by \autoref{lem: centre-lc-places}, and the latter together with \autoref{l-trick-dlt} yields point \autoref{itm:dltlc-minimal-modification:etale_dlt}.  
\end{proof}

\subsection{Computing local cohomology}

\begin{setting}\label{set-computation-local-cohomology}
For this subsection, let $(X, \Delta)$ be as in \autoref{s-lc-curve} and let $g \colon Z \to X$ be the birational modification constructed in \autoref{dltlc-minimal-modification}. 	Set $E:=\Ex(g)$ and $\Delta_Z:=g_*^{-1}\Delta+E$.
\end{setting}

We note that, by \autoref{dltlc-minimal-modification}.\autoref{itm:dltlc-minimal-modification:klt} $Z$ is $\bQ$-factorial and klt. Hence, by \autoref{c-rat-dlt}, for every $\bZ$-divisor $D$ on $Z$, the divisorial sheaf $\sO_Z(D)$ is Cohen--Macaulay. This also implies that any divisor on $Z$ is $(S_2)$ by \cite[Corollary 2.61]{kk-singbook}.

The following is the fundamental tool to relate the local cohomology group $H^2_x(X, \mathcal{O}_X)$ to cohomological properties of $g \colon Z \to X$.

\begin{lemma}\label{5spectral}
 There is an exact sequence as follows:
	$$0 \to H^0_x(X, R^1 g_* \MO_Z) \to H^2_x(X, \MO_X) \to H^2_{g^{-1}(x)}(Z, \MO_Z). $$
\end{lemma}
\begin{proof}
	Note that the composition of derived functors $R \Gamma_x$ and $Rg_*$ satisfies $R \Gamma_x \circ Rg_* = R \Gamma_{g^{-1}x}$. 
	Then we have the usual five-term short exact sequence:
	\begin{equation*}
	0 \to R^1 \Gamma_x R^0 g_* \sO_Z \to R^1 \Gamma_{g^{-1}(x)} \sO_Z  \to R^0 \Gamma_x R^1 g_* \sO_Z \to R^2 \Gamma_x R^0 g_* \sO_Z \to R^2 \Gamma_{g^{-1}(x)} \sO_Z.
    \end{equation*}
	To conclude, it is thus sufficient to show that $R^1\Gamma_{g^{-1}(x)}\mathcal{O}_Z$ vanishes.
	 By duality for CM sheaves (cf. \cite[Theorem 10.44]{kk-singbook} and \cite[Theorem 5.71]{km-book}), it is sufficient to show that $R^2g_*\sO_Z(K_Z)=0$. 
    As every irreducible component $F$ of $\Ex(g)$ surjects onto $C$, the fibres of $g$ are at most 1-dimensional and we deduce $R^2g_*\sO_Z(K_Z)=0$ by dimension reasons.
\end{proof}

We now prove a Grauert--Riemenschneider vanishing theorem for the birational contraction $g$.

\begin{proposition} \label{prop:GR-vanishing-lc}
Let $D$ be a $\mathbb{Z}$-divisor on $Z$ such that $D \sim_{\mathbb{Q},X} K_Z+\Delta'$, where $0 \leq \Delta' \leq g_*^{-1}\Delta$.
Then $R^ig_*\mathcal{O}_Z(D)=0$ for $i>0$.
\end{proposition}

\begin{proof}
For $i=2$, it is immediate as the dimension of the fibres of is at most 1 by \autoref{itm:dltlc-minimal-modification:lc} of \autoref{dltlc-minimal-modification}.
For the case $i=1$,
as $(Z, g_*^{-1}\lfloor{\Delta \rfloor}+E)$ is \'etale-dlt we can apply \autoref{seqlc} to the divisor $\big(D-K_Z-(m+1)E\big)$ for every $m\geq 0$ to obtain the short exact sequence of $\mathcal{O}_Z$-modules:
$$ 0 \to \mathcal{O}_Z(D-(m+1)E) \to \mathcal{O}_Z(D-mE) \to \mathcal{O}_{E} (G_m) \to 0,$$
where $G_m$ is a $\mathbb{Q}$-Cartier $\mathbb{Z}$-divisor on $E$.
Moreover, there exists a $\mathbb{Q}$-divisor $\Gamma_m$ such that $0 \leq \Gamma_m \leq \Diff_{E}(0)$ and
\begin{equation} \label{eq: Q-lin-equivalence}
\begin{aligned} 
G_m & \sim_{\bQ} K_{E} + \big(D- K_Z - (m+1)E\big)|_{E}  + \Gamma_m & \\
&\sim_{\bQ} (D - mE)|_{E} - (K_Z +E)|_{E} + K_{E} + \Gamma_m &  \\
& \sim_{\bQ} (D - mE)|_{E} - \Gamma_m' &
\end{aligned}
\end{equation}
where $\Gamma_m' := \Diff_{E}(0)-\Gamma_m \geq 0$.
Passing to cohomology, we obtain the short exact sequence
\begin{equation} \label{eq:ses-vanishing}
R^1g_*\mathcal{O}_Z(D-(m+1)E) \to R^1g_*\mathcal{O}_Z(D-mE) \to R^1g_*\mathcal{O}_{E}(G_m) \to 0.    
\end{equation}

We now claim that $R^1g_* \mathcal{O}_{E} (G_m)=0$ for all $m \geq 0$. 
By applying adjunction, we deduce
\begin{equation} \label{eq: adjunction-E}
    (D - mE)|_{E} \sim_{\mathbb{Q}} (K_Z + \Delta' - mE)|_{E} \sim_{\bQ} K_{E^\nu} + \Diff_{E}(\Delta') -( m+1)E|_{E}.
\end{equation}
Combining \autoref{eq: Q-lin-equivalence} and \autoref{eq: adjunction-E} we conclude that 
\begin{equation}
    G_m \sim_\bQ K_{E} +\Delta_m -(m+1)E|_{E},
\end{equation}
where $\Delta_m:=\Diff_{E}(\Delta')-\Gamma'_m \leq \Diff_{E}(\Delta')$.
Note that $\Delta_m=\big(\Diff_{E}(\Delta')-\Diff_{E}(0) \big)+\Gamma_m$ and therefore $\Delta_m \geq 0$. 
We verify we can apply \autoref{c-vanishing-slc-surf2} to $g \colon E \to C$ to show $R^1g_* \mathcal{O}_{E} (G_m)=0$ because 
\begin{itemize}
\item every log canonical centre of $(E, \Diff_{E}(\Delta'))$ dominates $C$ by \autoref{itm:dltlc-minimal-modification:surjects_onto_C} of \autoref{dltlc-minimal-modification};
\item $G_m \sim_\bQ K_{E}+\Delta_m +A$, where $0 \leq \Delta_m \leq \Diff(\Delta')$ and $A:=(-(m+1)E)|_{E}$ is $g$-nef by assumption;
\item there is an irreducible component $F$ of $E$ such that $A|_F$ is $g$-big. If this is not the case then, as the fibres of $g|_E$ are 1-dimensional, $-E_\eta$ is $g_\eta$-trivial and thus $E_\eta=0$ by the negativity lemma.
\end{itemize} 

Combining the vanishing $R^1g_* \mathcal{O}_{E} (G_m)=0$ with the sequence \autoref{eq:ses-vanishing}, it is sufficient to show $R^1g_*\mathcal{O}_Z(D-nE)=0$ for $n$ sufficiently large to conclude that $R^1g_*\mathcal{O}_Z(D)=0$ by descending induction.
As the pair $(Z, g_*^{-1}\Delta+(1-\varepsilon)E)$ is plt and 
$$K_Z +g_*^{-1}\Delta+(1-\varepsilon)E \equiv_X -\varepsilon E$$ 
is $g$-nef, we can consider the birational contraction $p \colon Z \to T$ to its canonical model $h \colon T \to X$. 
By construction, the pair $(T, h_*^{-1}\Delta+(1-\varepsilon)E_T)$ is plt and $-E \sim_{\bQ} p^*(-E_T)$ where $-E_T$ is ample over $X$. 
In particular, as 
$$ D-nE \sim_{\bQ} K_Z +\Delta' -nE \sim_{\bQ, T} K_Z+\Delta',$$ 
by \autoref{t-gr} we deduce 
$R^ip_*\mathcal{O}_Z(D-nE)=0$ for $i>0$. 
For $n$ sufficiently large and divisible, $nE \sim p^*nE_T$ and thus by the Leray spectral sequence and the projection formula we have $$R^1g_*\mathcal{O}_Z(D-nE) \simeq R^1h_*\big(p_*\mathcal{O}_Z(D-nE)\big) \simeq R^1h_*\big(p_*\mathcal{O}_Z(D) \otimes \mathcal{O}_T(-nE_T)\big),$$
which is zero for $n$ sufficiently large by Serre vanishing.
\end{proof}

As an application of the GR vanishing, we can finally compute the second local cohomology group at $x$.

\begin{proposition}\label{l-vanishing-h2-F-naive}
	The following equalities hold:
	\begin{enumerate}
		\item \label{itm:l-vanishing-h2-F-naive:H_2} $H^2_{g^{-1}(x)}(Z, \MO_Z)=0$;
		\item \label{itm:l-vanishing-h2-F-naive:H_1} $H^0_x(X, R^1g_*\MO_Z) \simeq H^0_x(C, R^1g_* \MO_E)$.
	\end{enumerate}
\end{proposition}

\begin{proof}
    To prove \autoref{itm:l-vanishing-h2-F-naive:H_2}, we note that $H^2_{g^{-1}(x)}(Z, \MO_Z)=R^1 g_* \MO_Z(K_Z)_x$ by local duality for Cohen--Macaulay sheaves and $R^1 g_* \MO_Z(K_Z)_x$ vanishes by \autoref{prop:GR-vanishing-lc}.
	
	To prove \autoref{itm:l-vanishing-h2-F-naive:H_1}, as $-E \sim_{\mathbb{Q},X} K_Z+g_*^{-1}\Delta$ we can apply \autoref{prop:GR-vanishing-lc} to deduce $R^ig_*\mathcal{O}_Z(-E)=0$ for $i>0$.
	Then the long exact sequence of cohomology associated to
	$0 \to \mathcal{O}_Z(-E) \to \mathcal{O}_Z \to \mathcal{O}_E \to 0$ implies that $R^1g_*\MO_Z\cong R^1g_*\MO_E$.
	If $i \colon C \to X$ denotes the closed immersion, the the equality $\Gamma_{C, x}=\Gamma_{X,x} \circ i_*$ holds, which implies $H^0_x(X, R^1 g_*\MO_E)\simeq H^0_x(C, R^1g_* \MO_E)$. 
\end{proof}

\begin{proof}[Proof of \autoref{t-hyp-1-good-model}]
This is a consequence of \autoref{dltlc-minimal-modification} and  \autoref{l-vanishing-h2-F-naive}. 
\end{proof}	

\section{$(S_2)$-condition for locally stable families of surfaces} \label{torsion}

In this section, we prove the $(S_2)$-conjecture for locally stable families of surfaces in characteristic $p \neq 2, 3$ and $5$. A proof of this theorem also appears in \cite[Corollary 23]{Arv23}.
In \autoref{S3-cond-lc}, we use this result to show the properness of $\overline{\mathcal{M}}_{2,v}$ over $\mathbb{Z}[1/30]$, contingent upon the existence of semi-stable reduction for family of stable surfaces in positive and mixed characteristic.

\subsection{Wild fibres}

\begin{setting}
    In this section, $(R, \mathfrak{m})$ is a DVR of perfect residue field $k=R/\mathfrak{m}$.
    We denote by $C$ the spectrum of $R$ and $x$ is its closed point.
    Given a morphism $f \colon S \to C$, we denote by $S_x$ the fibre over $x$.
\end{setting}

The terminology of wild fibres was introduced by Bombieri and Mumford in \cite{BM77} to study elliptic surface fibrations.
We present a more general definition for fibrations of surfaces over curves and we collect some foundational results proven by Raynaud in \cite{Ray70}.

\begin{definition}
	Let $S$ be a reduced connected surface and let $f \colon S \to C$ be a proper flat morphism such that $f_*\mathcal{O}_S=\mathcal{O}_C$.
	Consider the decomposition $$R^1f_*\MO_S=\mathcal{M} \oplus \mathcal{T},$$
	where $\mathcal{M}$ is a locally free sheaf of rank $\dim_{k(C)} H^1(S_{k(C)}, \mathcal{O}_{S_{k(C)}})$ and $\mathcal{T}$ is torsion sheaf supported at $x$.
	If $\mathcal{T}_x \neq 0$, we say 
 that the schematic fibre $f^{-1}(x)$ is a \emph{wild fibre} of $f$. 
\end{definition}

Given a proper flat morphism $f \colon X \to Y$, we say it is \emph{cohomologically flat in degree 0} if for any morphism $ g \colon Y' \to Y$ inducing the base change $f' \colon X':=X \times_Y Y' \to Y'$, then the canonical homomorphism of $\mathcal{O}_{Y'}$-modules
$ g^*f_*\mathcal{O}_X \to f'_* \mathcal{O}_{X'} $ is an isomorphism (see \cite[8.3.10]{FGA}).
    
\begin{lemma}\label{l-wild-cohomologically}
    Let $S$ be a reduced connected surface and let $f \colon S \to C$ be a proper flat morphism such that $f_*\mathcal{O}_S = \mathcal{O}_C$.
    If $x$ is a closed point of $C$, then 
    \[ \mathcal{T}_x \neq 0 \Leftrightarrow \dim_{k(x)} H^0(f^{-1}(x), \mathcal{O}_{f^{-1}(x)}) \geq 2. \]
    In particular, a wild fibre is not  reduced. 
    Moreover, $\mathcal{T} \neq 0$ if and only if $f$ is not cohomologically flat in degree 0.
\end{lemma}
\begin{proof}
    For each $i \geq 0$, consider the natural homomorphism of $k(x)$-vector spaces obtained from the base change $\Spec\big(k(x)\big) \to C$: 
    $$\alpha^i(x) \colon R^if_*\mathcal{O}_S \otimes k(x) \to H^i(f^{-1}(x), \mathcal{O}_{f^{-1}(x)}).$$
    As the fibres of $f$ have dimension 1, clearly $\alpha^2(x)$ is surjective and $R^2f_*\mathcal{O}_S=0$.
    Thus we deduce that $\alpha^1(x)$ is surjective by cohomology and base change for proper morphism \cite[Corollary 8.3.11.b]{FGA}.
    Applying once more \cite[Corollary 8.3.11.b]{FGA}, we deduce that $\alpha^0(x)$ is an isomorphism if and only if $\mathcal{T}=0$.
    To conclude, by hypothesis $f_*\mathcal{O}_S \otimes_{\mathcal{O}_C}  k(x) \simeq k(x)$ and therefore $\alpha^0(x)$ is an isomorphism if and only if $\dim_{k(x)} H^0(f^{-1}(x), \mathcal{O}_{f^{-1}(x)}) = 1$.
    
    Note that if $f^{-1}(x)$ is wild, then $\dim_{k(x)} H^0(f^{-1}(x), \mathcal{O}_{f^{-1}(x)}) \geq 2$. As $f^{-1}(x)$ is geometrically connected, we conclude $f^{-1}(x)$ is not reduced.
    The final assertion is shown in \cite[Corollary 8.3.11.a]{FGA}.
\end{proof}

A more precise characterisation of wild fibres for $(S_2)$-surfaces was proven by Raynaud \cite{Ray70} while investiganting representability criteria for Picard schemes of proper schemes over a DVR.

\begin{proposition}\label{p-raynaud}
    Let $f \colon S \to C$ be a proper flat morphism such that $f_*\mathcal{O}_S = \mathcal{O}_C$.
    Suppose that
    \begin{enumerate}
        \item $S$ is an $(S_2)$-surface such that its non-normal locus dominates $C$;
        \item the greatest common denominator of the multiplicities of the geometric special fibre $S_{\overline{k}}:=S_k \times_k \overline{k}$ is equal to 1.
    \end{enumerate}
    Then $f$ is cohomologically flat in degree 0 and $\mathcal{T}=0$.
\end{proposition}

\begin{proof}
    To verify the statement we can pass to a strict henselianisation $A^{\text{sh}}$ of $A$.
    The hypothesis guarantee that $S$ satisfies assumption (N)$^{*}$ of \cite[Definition 6.1.4]{Ray70}. Indeed, $S_x$ is $S_1$ as it is a Cartier divisor on the $(S_2)$-surface $S$. Moreover, at every generic point $\eta$ of an irreducible component of $S_x$ we have that $\mathcal{O}_{S,\eta}$ is regular.   
    Then the statement is proven in the implication $(i) \Rightarrow (iv)$ of  \cite[Theorem 8.2.1]{Ray70}.
\end{proof}

We recall that no wild fibres appear when the generic fibre of $f \colon S \to C$ is a tree of conic curves.

\begin{proposition}\label{p-genus0}
Let $f \colon S \to C$ be a proper morphism onto a regular curve such that $f_*\mathcal{O}_S =\mathcal{O}_C$.
Suppose that $S$ is an $(S_2)$-surface such that
each of its irreducible components $S_i$ dominates $C$ and the non-normal locus of each $S_i$ dominates $C$.
If $H^1(S_{k(C)}, \mathcal{O}_{S_k(C)})=0$, then $\mathcal{T}=0$.   
\end{proposition}
 
\begin{proof}
As every irreducible component of $S$ dominates $C$, the morphism $f$ is equi-dimensional. 
As $S$ has dimension 2, it is Cohen--Macaulay and, as $C$ is regular, $f$ is flat by miracle flatness \cite[\href{https://stacks.math.columbia.edu/tag/00R4}{Tag 00R4}]{stacks-project}.
We can thus apply \cite[Proposition 9.3.1]{Ray70}.
\end{proof}

\begin{remark}
    The conditions imposed on $S$ in \autoref{p-raynaud} and \autoref{p-genus0} are optimal as shown in the examples of \cite[Section 9]{Ray70}.
\end{remark}

If the generic fibre has arithmetic genus at least 1, then wild fibres can appear when the residue field has characteristic $p>0$. 

\begin{example}\label{e-wild}
    Suppose $k$ is algebraically closed of characteristic $p>0$.
    We recall the construction of wild fibres explained by Raynaud \cite{Ray70}. 
    Let $E_{k(C)}$ be an elliptic curve (ordinary if $\characteristic(k(C))>0$) such that the special fibre of its Néron model is either a supersingular elliptic curve $E_{k}$ or the multiplicative group $\mathbb{G}_{m,k}$.
    Let $S_{k(C)}$ be a regular torsor over $E_{k(C)}$ of order $p^n$ for $n>0$ and let $f \colon S \to C$ be its minimal model. Then $\mathcal{T} \neq 0$ by \cite[Théorème 9.4.1.b]{Ray70}.
    
    Another set of examples, based on Artin--Schreier coverings, is discussed in \cite[Section 8]{KU85}.
\end{example}


	

\subsection{Cohen--Macaulay criteria for log canonical 3-fold singularities} \label{S3-cond-lc}

Throughout this section, we suppose $(X, \Delta)$ is 3-dimensional log canonical singularity as in \autoref{setting:runing_MMP}. 

Note that, as proved in \cite[Theorem 1]{AP23}, the 1-dimensional scheme $C$ is normal and thus regular.
Using \autoref{t-hyp-1-good-model} the failure of Cohen--Macaulayness is explained by the presence of a wild fibre on a proper birational modification.
We use this to show that if the surface singularity at the generic point of $C$ is rational, then $X$ is Cohen--Macaulay.

\begin{proposition}\label{p-rational-cod1}
    Let $C \subset (X,\Delta)$ as in \autoref{s-lc-curve}.
    If $\mathcal{O}_{X, \eta}$ is a rational surface singularity, then $X$ is $(S_3)$ at $x$.
    In particular, if $C \subset \Supp(\Delta)$, then $X$ is $(S_3)$ at $x$.
\end{proposition}

\begin{proof}
    Let $g \colon (Z, g_*^{-1}\Delta+E) \to (X, \Delta)$ be the crepant proper birational morphism constructed in \autoref{t-hyp-1-good-model}.
    It is sufficient to show that $H^0_x(E, R^1(g|_E)_* \mathcal{O}_E)$ vanishes to conclude.
    Note that $E$ is $(S_2)$.
    Moreover, the irreducible components of $E$ and their non-normal loci dominate $C$ as it is the minimal log canonical   centre passing through $x$.
    As $\mathcal{O}_{X, \eta}$ is a rational singularity and $C$ is regular, we apply \autoref{p-genus0} to conclude $H^0_x(C, R^1g_*\mathcal{O}_E)=0$.
    For the last assertion, we just observe that $\mathcal{O}_{X,\eta}$ is a rational singularity by \cite[Proposition 2.28]{kk-singbook}.
\end{proof}
   
\begin{corollary}\label{c-wild}
    Let $C \subset (X,\Delta)$ as in \autoref{s-lc-curve}.
    Let $g \colon (Z, g_*^{-1}\Delta+E) \to (X,\Delta)$ be the modification constructed in \autoref{t-hyp-1-good-model}.
    If $H^2_x(X, \mathcal{O}_X) \neq 0$, then 
    \begin{enumerate}
        \item $\mathcal{O}_{X, \eta}$ is not rational and $C$ is not contained in $\Supp(\Delta)$, 
        \item the fibre $E_x$ is wild.
    \end{enumerate}
\end{corollary}

\begin{proof}
    (a) is proven in \autoref{p-rational-cod1}. 
    As $h^1(E_{\eta}, \mathcal{O}_{E_{\eta}}) \neq 0$ and $\mathcal{O}_{X, \eta}$ is log canonical, then $\deg_{\eta} \omega_{E_\eta}=0$.
    By \autoref{t-hyp-1-good-model}.\autoref{itm: id_h2_h0}  we have $H^0_x(C, R^1g_*\mathcal{O}_E) \simeq H^2_x(X, \mathcal{O}_X) \neq 0$ and thus $E_x$ is a wild fibre for $g \colon E \to C$.
\end{proof}
   
As a byproducts of the results of \autoref{s-depth-3fold} and the properties of wild fibres \autoref{p-raynaud}, we can prove the $(S_3)$-condition of log canonical 3-folds pairs in the case where a Cartier divisor is an addendum of the boundary divisor.
This answers \cite[Question  8]{Kol22} affirmatively if the characteristic of the residue field is different from $2, 3$ and $5$.

\begin{theorem}\label{t-s3-under-dlt-modification}
	Let $(X,X_0+\Delta)$ be a 3-dimensional log canonical pair and let $x$ be a closed point of $X$.
	Suppose $X_0$ is a non-zero effective Cartier divisor such that $x \in X_0$.
	Then $X$ is $(S_3)$ at $x$.
\end{theorem}

\begin{proof}
We can localise at the closed point $x$.
As $X_0$ is effective, $x$ cannot be a minimal log canonical centre for $(X,\Delta)$.
By \autoref{l-not1d-minimal-lc} and \autoref{c-wild}, we can suppose that
\begin{enumerate}
    \item the minimal log canonical centre $C$ of $(X, X_0+ \Delta)$ passing through $x$ has dimension 1;
    \item if $\eta$ is the generic point of $C$, then $\mathcal{O}_{X, \eta}$ is a non-rational surface singularity and $\eta \notin \Supp(X_0+\Delta)$.
\end{enumerate}
If $g \colon (Z, g_*^{-1}\Delta+E) \to (X, \Delta)$ is the crepant proper birational morphism constructed in \autoref{t-hyp-1-good-model}, it is sufficient to show that $H^0_x(E, R^1(g|_E)_* \mathcal{O}_E)$ vanishes to conclude. 
For this we argue by contradiction. 

First we note that the pair $(Z, g_*^{-1}\Delta+E+g^*X_0)$ is log canonical and $E$ is $(S_2)$. By adjunction \cite[Lemma 4.8]{kk-singbook} the  pair $(E, \Diff_E(g^*X_0+g_*^{-1}\Delta))$ is slc and, as $X_0$ is Cartier, we have $\Diff_E(g^*X_0+g_*^{-1}\Delta)=(g^*{X_0})|_E+\Diff_E(g_*^{-1}\Delta)$ by \cite[Lemma 2.5]{kk-singbook}.
As $E_x$ is a wild fibre, by \autoref{p-raynaud} each of its irreducible components is non-reduced. 
As $X_0$ is an effective Cartier divisor not containing $C$, then $(g^*X_0)|_E$ must have coefficients strictly larger than 1, contradicting that $(E, \Diff_E(g^*X_0+g_*^{-1}\Delta))$ is slc. 
\end{proof}

\subsection{Properness of the moduli space of stable surfaces}\label{ss-properness}
		
We briefly recall the natural set-up for the study of stable and locally stable families and we refer to \cite[Chapter 2]{k-moduli} for a thorough discussion.
	
Let $C=\Spec(R)$, where $(R, \mathfrak{m})$ is a DVR with perfect residue field $k \coloneqq R/\mathfrak{m}$ of characteristic $p>0$, and fraction field $K \coloneqq \Frac(R)$. 
We say that a morphism  $f\colon X\to C$ is \emph{family of varieties} is $f$ is a flat morphism of finite type such that for every $c \in C$ the fiber $X_c$ is pure dimensional, geometrically reduced and geometrically connected.  
We denote the special (resp. generic) fibre of $f$ by $X_k$ (resp. $X_K$). 
A \emph{family of pairs} is  $f \colon (X,\Delta) \to C$ is a family of varieties $f\colon X \to C$ together with an effective Mumford $\Q$-divisor $\Delta$ on $X$ such that $\Supp(\Delta)$ does not contain any irreducible components of $X_k$ and none of the irreducible components of $X_k \cap \Supp(\Delta)$ is contained in $\textup{Sing}(X_k)$. 

	\begin{definition}\label{def-loc-stable}
	We say $f \colon (X, \Delta) \to C$ is a \emph{locally stable} (or \emph{slc}) \emph{family} if $f$ is a family of pairs and $(X, \Delta+X_k)$ has slc singularities. 
	We say $f$ is a \emph{stable family} if $f$ is a projective locally stable family such that $K_X+\Delta$ is ample over $C$.
	\end{definition}
 
	In \cite[Corollary 10.2]{7authors}, $\overline{\mathcal{M}}_{2,v}$ is shown to exist as a separated Artin stack of finite type over $\mathbb{Z}[1/30]$ with finite diagonal. 
    The main open question on $\overline{\mathcal{M}}_{2,v}$ is whether it is a proper stack over $\mathbb{Z}[1/30]$ (some cases are discussed in \cite[Theorem 10.6]{7authors}).
	
	To prove properness, one has to prove the valutative criterion for families of stable surfaces. 
	As explained in \cite[Section 6]{Pos21}, this can be reduced to two problems on locally stable families: the existence of a locally-stable reduction and the $(S_2)$-condition on stable limits. We recall their precise formulation. 
 
	\begin{enumerate}
	    \item [\textbf{(LSR)}] Let $X \to C$ be a flat projective morphism where $X$ is a regular 3-fold. Let $E$ be a reduced effective divisor on $X$ such that $(X, E + (X_k)_{\red})$ is snc for every closed point $c \in C$. 
	    Then there exists a finite morphism $C' \to C$ such that: if $Y$ is the normalization of $X \times_C C'$ and $E_Y$ is the pull-back divisor, then every closed fiber $Y_{k'}$ is reduced and every $(Y, E_Y + Y_{k'})$ is log canonical.
	    \item [\textbf{(S$_2$)}] Let $(X, \Delta) \to C$ be a stable family of surface pairs.
	    Then $X_k$ is $(S_2)$.
	\end{enumerate}
	
In equicharacteristic 0, existence of semi-stable reduction is proven in \cite{KKMS} (see also \cite[Theorem 7.17]{km-book}) and the $(S_2)$-property is proven in \cite{Ale08} (see also \cite[Definition–Theorem 2.3]{k-moduli}).
While semi-stable reduction of surfaces is still an open conjecture, the results of \cite{BK20} can be used to prove the $(S_2)$-condition for the closure of the locus of klt stable varieties (see the last lines of the proof of \cite[Theorem 10.6]{7authors}).
We now settle the general semi-log canonical case. 

\begin{theorem}\label{t-S2-property}
Suppose $p \neq 2,3$ and $5$.
If $f \colon (X, \Delta) \to C$ is a stable family of surfaces, then $X_k$ is $(S_2)$ and $(X_k, \Diff_{X_k}(\Delta))$ is slc.
\end{theorem}
	
\begin{proof}
	If $X$ is normal, then $(X,\Delta+X_k)$ is a log canonical pair. For every closed point $p \in X_k$, the local ring $\mathcal{O}_{X,p}$ is $(S_3)$ by \autoref{t-s3-under-dlt-modification}. As $X_k$ is a Cartier divisor, we deduce $X_k$ is $(S_2)$ by \cite[Corollary 2.61]{kk-singbook}. By performing adjunction \cite[Definition 4.2]{kk-singbook}, we deduce that the normalisation $\big(X^{\nu}_k, \Diff_{X^{\nu}_k}(\Delta)\big)$ is log canonical by \cite[Lemma 4.8]{kk-singbook}. 
    Therefore $(X_k, \Diff_{X_k}(\Delta))$ is semi-log canonical by definition.
    
    Suppose $X$ is demi-normal and let $\pi \colon Y \to X$ be its normalisation. We write $K_Y+D+\pi^*\Delta=\pi^*(K_X+\Delta)$, where $D$ is the divisorial part of the conductor.
    Then $(Y,D+\pi^*\Delta) \to C$ is a stable family of pairs, where $Y$ is normal. By the previous step, $Y_k$ is $(S_2)$ and the pair $(Y_k, \Diff_{Y_k}(\pi^*\Delta))$ is slc.  
    We conclude $X_k$ is $S_2$ and $(X_k, \Diff_{X_k}(\Delta))$ is slc by \cite[Proposition 4.2.6]{Pos21b}.
	\end{proof}

We now have all the ingredients to prove the main result of this article.

\begin{theorem}\label{cor: properness_moduli}
Assume (LSR).
Then the moduli stack $\overline{\mathcal{M}}_{2,v}$ is proper over $\mathbb{Z}[1/30]$ and the coarse moduli space $\overline{M}_{2,v}$ is projective over $\mathbb{Z}[1/30]$.
\end{theorem}

\begin{proof}
The proof of \cite[Theorem 6.0.5]{Pos21} works also in mixed characteristic and thus the (LSR) hypothesis together with \autoref{t-S2-property} conclude the properness of $\overline{\mathcal{M}}_{2,v}$.
The projectivity of $\overline{M}_{2,v}$ is then shown in \cite[Theorem 1.2]{Pat17}.
\end{proof}

We conclude by giving an application to the asymptotic invariance of plurigenera for log canonical minimal surface pairs of general type.
This generalises the klt case proven in \cite[Theorem 4.1]{BBS22}.

\begin{corollary} \label{cor: asymptotic-invariance}
    Suppose $p >5$.
    Let $(X,\Delta)$ be a 3-dimensional 
    pair and 
    let $\pi \colon (X, \Delta) \to C$ be a projective contraction such that $(X_k,\Diff_{X_k}(\Delta))$ is log canonical. 
    If $K_X+\Delta$ is nef and big over $C$, then there exists $m_0>0$ such that 
    $$ \dim_k H^0(X_k, m(K_{X_k} + \Delta_k)) = \dim_K H^0(X_K, m(K_{X_K} + \Delta_K)) \text{ for all } m \in m_0 \mathbb{N}$$
\end{corollary}

\begin{proof}
    By inversion of adjunction \cite[Corollary 10.1]{7authors}, we conclude $(X,X_k+\Delta)$ is a log canonical pair.
    We first claim that $K_X+\Delta$ is semi-ample over $C$. Since the characteristic of the residue field of $R$ is $p>0$, by \cite[Theorem 2.2]{Wit21} it is sufficient to check semiampleness fiber-wise. 
    For this, note that $K_{X_K}+\Delta_K$ and $(K_{X}+\Delta)|_{X_k} \sim_{\mathbb{Q}} K_{X_k}+\Diff_{X_k}(\Delta)$ are semiample by the abundance theorem for log canonical surfaces \cite{Tan20}.
    Let $f \colon (X,\Delta) \to (Z, \Delta_Z=f_*\Delta)$ be the semiample birational contraction associated to $K_X+\Delta$. Note that $f \colon (X, X_k+\Delta) \to (Z, Z_k+\Delta)$ is also crepant.
    
    By \cite[Lemma 2.17]{BBS22}, it is sufficient to check that $(f_k)_*\mathcal{O}_{X_k}=\mathcal{O}_{Z_k}$ to conclude.
    By considering the Stein factorisation $f_k \colon X_k \to Y \to Z_k$, we are left to show $g \colon Y \to Z_k$ is an isomorphism.
    As the morphism $g \colon Y \to Z_k$ is a birational morphism and $Y$ is normal, it is sufficient to verify that $Z_k$ is normal to conclude.
    By construction $(Z,Z_k+\Diff_{Z}(\Delta_Z))$ is log canonical and thus $Z_k$ satisfies the $(S_2)$ condition by \autoref{t-S2-property}.
    By Serre's criterion for normality, we are  thus left to show that $Z_k$ is $(R_1)$ and we argue by contradiction.
    Suppose there exists a codimension 1 point $\eta$ of $Z_k$ such that $Z_k$ is not normal. Then by inversion of adjunction $\eta$ is the generic point of a log canonical centre of $(Z,Z_k+\Delta)$ and thus it is nodal. 
    By \cite[Lemma 2.7]{Bri22} the normalisation of $Z_k$ is a universal homeomorphism and thus we conclude.
    \end{proof}

\begin{remark}
    In \cite[Theorem 1]{Kol22}, Koll\'{a}r proves that the moduli space of stable 3-folds is not proper over any field of characteristic $p>0$.
    In particular, \cite[Example 4]{Kol22} show that \autoref{t-S2-property}, \autoref{cor: properness_moduli} and \autoref{cor: asymptotic-invariance} do not generalise to dimension 3, even for large $p$.
\end{remark}

\begin{question}
    We leave open the question whether \autoref{t-S2-property} hold in characteristic $p \leq 5$. 
    Note that the examples of non-normal plt centres constructed \cite{CT19} are not Cartier.
\end{question}

\section{Counterexamples to local Kawamata--Viehweg vanishing} \label{s-counterexample}

	We conclude by constructing a counterexample to the local Kawamata--Viehweg vanishing theorem for log canonical 3-dimensional singularities in positive and mixed characteristic (\autoref{t-counterexample}).  
	The counterexample is obtained by taking the relative cone over an elliptic surface fibration with a wild fibre.

\subsection{Relative cone construction}
	
	In this section we develop the theory of relative cones, expanding on \cite[Section 3.2]{kk-singbook}.
	Let $f \colon X \to T$ be a projective flat morphism of normal integral schemes with $f_* \MO_X=\MO_T$ and
	let $L$ be an $f$-ample invertible sheaf. We define the affine $T$-scheme:
	
	$$C_a(X, f, L)= \Spec_T \bigoplus_{m \geq 0} f_* \mathcal{O}_X(mL) \to T. $$
	
	The scheme $C_a(X,f,L)$ is the \emph{relative cone} of $f$ with respect to $L$.
	The natural subscheme $V_T \subset C_a(X,f, L)$ defined by by the ideal sheaf $ \bigoplus_{m \geq 1} f_* \mathcal{O}_X(mL)$ is isomorphic to $T$ and it is called the \emph{relative vertex}.
	The variety $C_a^*(X,f,L)=C_a(X, f, L) \setminus V_T \simeq BC_a^*(X, f, L) \setminus E$ is called the \emph{relative punctured cone}. 	
	
	We have the following commutative diagram 
	\[
	\xymatrix{
		BC_a(X, L):= \Spec_X \bigoplus_{m \geq 0} L^m \ar[d]_\pi \ar[r]^{\qquad \qquad p} &C_a(X, f, L) \ar[d]\\
		X \ar[r]^f & T,
	}
	\]
	where $p$ is a birational projective morphism with exceptional divisor $E \simeq X$ onto $V_T$ such that $$\mathcal{O}_{BC_a(X,f,L)}(E)|_E=L^{\vee}.$$
	
	The following is a generalisation of \cite[Proposition 3.14]{kk-singbook} to the relative setting.
	\begin{proposition} \label{p-properties-cones}
		With the same setting as above, we have
		\begin{enumerate}
			\item[(i)]	$\Pic(C_a(X, f, L)) \simeq \Pic(T)$, 
			\item[(ii)] $\Cl(C_a(X, f, L)) \simeq \Cl(X)/\langle L \rangle.$
		\end{enumerate}
		Let $\Delta_X$ be a $\Q$-divisor on $X$ and assume $K_X+\Delta_X$ is $\Q$-Cartier.
		We define $\Delta_{BC_a(X,L)}=\pi^*\Delta$ and $\Delta_{C_a(X,f,L)}=p_*\Delta_{BC_a(X,L)}$.
		We have the following
		\begin{enumerate}
			\item[(iii)] $K_{BC_a(X,L)}+\Delta_{BC_a(X,L)}+E \sim_{\mathbb{Q}} \pi^*(K_X+\Delta)$,
			\item[(iv)] $m(K_{C_a(X, f, L)}+\Delta_{C_a(X, f, L)})$ is Cartier iff $m(K_X+\Delta) \sim_{f} L^{rm} $ for some $r \in \Q$.
			In this case we have $$K_{BC_a(X,L)}+\Delta_{BC_a(X,L)}+(1+r)E \sim_{\Q} p^*(K_{C_a(X, f, L)}+\Delta_{C_a(X, f, L)}). $$
		\end{enumerate}
	\end{proposition}
	
	\begin{proof}
		Since $BC_a(X,L)$ is an $\mathbb{A}^1$-bundle over $X$, we have $\Cl(BC_a(X,L))\simeq \Cl(X)$ and $\Pic(BC_a(X,L)) \simeq \Pic(X)$.
		Let us note we have the following commutative diagram:
		\[
		\xymatrix{
			\Pic(X) \ar[r]^{\pi^* \qquad} & \Pic(BC_a(X, f, L)) \ar[r] & \Pic(E) \\
			\Pic(T) \ar[r] \ar@{^{(}->}[u]_{f^*} & \Pic(C_a(X,f,L)) \ar[r] \ar@{^{(}->}[u]_{p^*}  & \Pic(V_T) \ar@{^{(}->}[u],
		}
		\]
  where the top arrows are  all isomorphisms.
		We prove (i). Let $D$ be an invertible sheaf on $C_a(X, f, L)$, then $p^*D|_E$ is the pull-back of a line bundle on $V_T$, thus proving (i). Items (ii) and (iii) are proven in \cite[Proposition 3.14]{kk-singbook}.
		Recall that, since $\Pic(C_a^*(X,f,L)) \hookrightarrow \Cl(C_a^*(X,f,L)) \simeq \Cl(X)/\langle L \rangle$, the kernel of the morphism $\pi|_{C_a^*(X,f,L)}^* \colon \Pic(X) \to \Pic(C_a^*(X,f,L))$ is the subgroup generated by $L$.
		
		As for (iv) $m(K_{C_a(X,f,L)}+\Delta_{C_a(X,f,L)})$ is Cartier if and only if the Weil divisor $m(K_{C_a^*(X,f,L)}+\Delta_{C_a^*(X,f,L)})$ is the pull-back of a Cartier divisor on $T$ by (i).
		In turn, this is equivalent to ask that $\pi|_{C_a^*(X,f,L)}^{*}(mK_X+m\Delta)=\pi|_{C_a^*(X,f,L)}^{*}f^* D$ for some Cartier divisor $D$ on $T.$
		This is equivalent to  $m(K_{X}+\Delta)-f^*D \sim L^{rm}$ for some $r \in \Q$, thus concluding the first part.
		As for the last equality, let us write
		$$K_{BC_a(X,L)}+\Delta_{BC_a(X,L)}+(1+a)E \sim_{\Q} p^*(K_{C_a(X, f, L)}+\Delta_{C_a(X, f, L)}). $$
		By restricting to $E$ we have $K_X+\Delta+aE|_E \sim_{\Q,f} 0, $ which becomes $rL -aL=0,$ thus $r=a$.
	\end{proof}
	
	As a corollary we have the following result on the singularities of relative cones (to compare with \cite[Lemma 3.1]{kk-singbook}).
	
	\begin{proposition}\label{p-relative-cones-mmpsing}
		In the previous setting, assume $K_X+\Delta$ is $\Q$-Cartier and $K_X+\Delta \sim_{f, \Q} rL$ for some $r \in \Q$. Then $(C_a(X,f,L),\Delta_{C_a(X,f,L)})$ is 
		\begin{enumerate}
			\item klt if $r<0$ and $(X,\Delta)$ is klt,
			\item log canonical if $r \leq 0$ and $(X,\Delta)$ is log canonical.
		\end{enumerate}
	\end{proposition}
	
	\begin{proof}
		By \autoref{p-properties-cones}, we have
		\begin{equation*} \label{eq: 1} 
			\begin{split}
				\discrep(&(C_a(X,f,L),\Delta_{C_a(X,f,L)}))=  \\
				&  =\min\left\{-(1+r), \text{discrep}(BC_a(X,L),\Delta_{BC_a(X,L)}+(1+r)E)\right\}.
			\end{split}
		\end{equation*} 
		Since $\pi$ is a smooth morphism and $E$ is a section for $\pi$ we conclude by \cite[2.14, Equation (4)]{kk-singbook} that
    \begin{equation*} 
        \discrep(BC_a(X,L),\Delta_{BC_a(X,L)}+E) = \totdiscrep(X,\Delta).
    \end{equation*}
    Thus (a) and (b) are automatic.    
\end{proof}
	
	\subsection{Failure of the ($S_3$)-condition at a non-minimal lc centre}
	
	We construct an example showing that \autoref{t-dpt-C} is not valid in general in positive and mixed characteristic, thus showing the statement of \autoref{t-s3-under-dlt-modification} is sharp.
	We fix $C$ to be the spectrum of a DVR whose closed point is perfect of characteristic $p>0$.
	\begin{proposition} \label{p-minimal-lc}
		Let  $S$ be a regular surface and let $f \colon S \to C$ be a minimal elliptic fibration together with a relatively $f$-ample invertible sheaf $L$.
		Then $X:= C_a(S,f,L)$ is a 3-dimensional log canonical singularity, the map $p \colon Y:=BC_a(S,L) \to X$ is a log resolution and the vertex $V_C$ is the unique log canonical centre of $X$.
	\end{proposition}
	\begin{proof}
		By \autoref{p-properties-cones}, we have
		$K_{Y}+E \sim_{\Q} p^* K_X. $
		Since $(Y, E)$ is log smooth and $E$ is irreducible we conclude that $E$ is the unique log canonical place of $X$. Thus $V_C$ is the unique log canonical centre of $X$.
	\end{proof}
	
    In the next lemmas we compute the local cohomology at a closed point in a log canonical place.
    
	\begin{lemma}\label{l-easy-van}
		Let $D$ be a $\Z$-divisor on $Y$. If $D \sim_{p, \Q} K_Y$, then $R^1p_*\MO_Y(D)=0$.
	\end{lemma}
	
	\begin{proof}
		By \cite[Theorem 3.3]{Tan18}, the  relative Kawamata--Viehweg vanishing theorem holds for the morphism $p|_E \colon E \to V_C$. 
		Since $-E$ is $p$-ample and $D \sim_{p, \Q} (K_Y+E)-E$ we conclude by \cite[Proposition 22]{BK20}.
	\end{proof}
	
	\begin{proposition} \label{p-second-loc-coh}
		Let $x \in V_C \simeq C$. Then
		$H^2_x(X, \mathcal{O}_X) \simeq H^0_x(C, \mathcal{T}).$ 
	\end{proposition}
	
	\begin{proof}
		By the Leray spectral sequence for local cohomology
		$H^i_x(X, R^j p_* \MO_Y) \Rightarrow H^{i+j}_{p^{-1}(x)} (Y, \mathcal{O}_Y), $
		we have the exact sequence
		$$H^0_x(X, R^1 p_* \MO_Y) \to H^2_x(X, \MO_X) \to H^2_{p^{-1}(x)}(Y, \MO_Y). $$
		
		\begin{claim}\label{c-iso}
		The following isomorphisms hold:
			\begin{enumerate}
				\item[(i)] $H^2_{p^{-1}(x)}(Y, \MO_Y) =0$,
				\item[(ii)] $R^1p_*\MO_Y \simeq R^1(p|_E)_* \MO_E$.
			\end{enumerate}
		\end{claim}
		\begin{proof}
			To prove $(i)$, as $Y$ is regular we can apply duality \cite[10.44]{kk-singbook} to deduce $H^2_{p^{-1}(x)}(Y, \MO_Y)\simeq (R^1p_*\MO_Y(K_Y))_x$ and we conclude by \autoref{l-easy-van}.
			
			To prove $(ii)$, it is enough to show $R^1 p_* \MO_Y(-E)=0$ as $R^2\pi_* \MO_Y(-E)=0$ since the fibers of $p$ are at most 1-dimensional. 
			Since $-E \sim_{p, \Q} K_Y$ we conclude again by \autoref{l-easy-van}.
		\end{proof}
		Using \autoref{c-iso} we have $$H^2_x(X, \MO_X) \simeq H^0_x(X, R^1 p_* \MO_Y) \simeq H^0_x(X, R^1(p|_E)_* \MO_E)$$
		We denote by $i \colon V_C \to X$ the natural injection. Then we have the following isomorphism $H^0_x(X, R^1(p|_E)_* \MO_E)\simeq H^0_x(X, i_*(\mathcal{M} \oplus \mathcal{T})).$ 
		Using again the Leray spectral sequence of local cohomology for $i$ we have
		$H^0_x(X, i_*(\mathcal{M} \oplus \mathcal{T})) \simeq H^0_x(C, \mathcal{M} \oplus \mathcal{T})=H^0_x(C, \mathcal{T}),$
		since $\mathcal{M}$ is locally free thus concluding.
	\end{proof}
	
\begin{proof}[Proof of \autoref{t-counterexample}]
	Let $f \colon S \to C$ be a minimal elliptic fibration such that $\mathcal{T} \neq 0$ (such a surface exists by \autoref{e-wild}).
	By \autoref{p-minimal-lc}, $X=C_a(S,f,L)$ is a 3-dimensional log canonical variety, where the relative vertex $V_C$ is the unique minimal log canonical centre.
	By \autoref{p-second-loc-coh}, we deduce that the local cohomology $H^2_x(X, \MO_X) \neq 0,$ thus proving (c).
\end{proof}

\begin{question}
In \autoref{t-counterexample} we construct a log canonical 3-fold singularity $X$ with a minimal 1-dimensional log canonical centre and not CM, showing optimality of \autoref{p-rational-cod1} and \autoref{t-s3-under-dlt-modification} in the case where the exceptional divisor $E_\eta$ is a a regular curve of genus 1. 
We do not if the failure of Cohen--Macaualyness can appear in the case where $E_\eta$ is a nodal curve.
\end{question}

\begin{remark}
    A guiding principle in birational geometry in characteristic $p$ says that properties of klt and dlt singularities should behave similarly to characteristic 0 if $p$ is sufficiently large compared to the dimension (\cite[Section 6]{Tot19}).
    This principle does not apply to log canonical singularities.
    For example, in \cite[Corollary 6]{Kol22} Koll\'ar shows examples of 4-dimensional log canonical pairs with non-weakly normal lc centre in every characteristic $p>0$.
    \autoref{t-counterexample} shows pathological phenomena already appear in dimension 3 for every prime number.
\end{remark}
 
	\bibliographystyle{amsalpha}
	\bibliography{refs}
	
\end{document}